\documentclass[12pt]{amsart}

\usepackage{amsmath,amsthm,mathtools}
\usepackage{enumerate,geometry,multicol,graphicx,caption}
\usepackage{enumitem}
\usepackage{color,soul}
\usepackage[dvipsnames]{xcolor}
\usepackage{tikz}
\usepackage{ytableau}
\usepackage{afterpage}

\usepackage{marginnote}
\setlist[itemize]{leftmargin=2em}

\usepackage[charter]{mathdesign}
\usepackage{dsfont}

\usepackage[
	style=alphabetic,
	sorting=nyt,
	useprefix=false,
	maxbibnames=99, maxalphanames=99,
	backend=bibtex
]{biblatex}
\AtBeginBibliography{\small}
\DeclareFieldFormat{pages}{#1}  

\DeclareNameAlias{sortname}{family-given}
\DeclareNameAlias{default}{family-given}
\DeclareNameFormat{family-given}{%
	\ifnum\value{listcount}=1
	    \usebibmacro{name:family-given}
		{\namepartfamily,}
		{\namepartgiveni}
		{\namepartprefix}
		{\namepartsuffix}%
	\else
		\usebibmacro{name:given-family}
		{\namepartfamily}
		{\namepartgiveni}
		{\namepartprefix}
		{\namepartsuffix}%
	\fi
\usebibmacro{name:andothers}}

\usepackage{hyperref}
\hypersetup{
colorlinks=true,
linkcolor=teal,
filecolor=teal,
urlcolor=teal,
citecolor=teal
}
\usepackage{thmtools}
\numberwithin{equation}{section}
\declaretheorem[name=Theorem, numberlike=equation]{theorem}
\declaretheorem[name=Lemma, numberlike=equation]{lemma}

\declaretheorem[name=Corollary, numberlike=equation]{corollary}

\declaretheorem[name=Fact, numberlike=equation]{fact} 

\declaretheorem[name=Theorem, numbered=no]{theorem*}
\declaretheorem[name=Main Theorem, numbered=no]{mainthm*}
\declaretheorem[name=Corollary, numbered=no]{corollary*}
\declaretheorem[name=Conjecture, numbered=no]{conjecture*}
\declaretheorem[name=Acknowledgements, numbered=no]{acknowledgements*}

\declaretheorem[name=Definition, numberlike=equation, style=definition]{definition}
\declaretheorem[name=Example, numberlike=equation, style=definition]{example}
\declaretheorem[name=Example, numbered=no, style=definition]{example*}
\declaretheorem[name=Question, numberlike=equation, style=definition]{question}
\declaretheorem[name=Remark, numberlike=equation, style=definition]{remark}
\declaretheorem[name=Theorem, numberlike=equation, style=definition]{thm}

\usepackage[noabbrev]{cleveref}
\crefname{thm}{Theorem}{Theorems}
\crefname{fact}{Fact}{Facts}

\renewcommand{\mathbb}{\mathds}

\DeclareMathOperator{\Asc}{Asc}
\DeclareMathOperator{\col}{col}

\DeclareMathOperator{\Fl}{Fl}
\DeclareMathOperator{\Gr}{Gr}
\DeclareMathOperator{\smooth}{smooth}
\DeclareMathOperator{\SYT}{SYT}
\DeclareMathOperator{\syt}{SYT}

\renewcommand{\sl}{{\mathfrak{sl}}}
\newcommand{\Sn}{{\mathfrak{S}_n}}
\newcommand{\SYTs}{\SYT_{\smooth}}
\newcommand{\syts}{\SYTs}

\title{Webs and smooth components of two column Springer fibers}
\author{Mike Cummings}
\date{\today}

\address{Dept. of Combinatorics and Optimization, University of Waterloo, Waterloo ON N2L 3G1, Canada}
\email{\href{mailto:mike.cummings@uwaterloo.ca}{mike.cummings@uwaterloo.ca}}

\begin{document}

\begin{abstract}
	Webs and Springer fibers are separately important objects in representation theory: webs give a diagrammatic calculus for tensor invariants of $\sl_k$, and the cohomology groups of Springer fibers can be used to construct the irreducible representations of the symmetric group.
	Fung's 1997 thesis gave the first evidence of a connection between $\sl_2$ webs and Springer fibers, showing that webs naturally index and describe the components of certain ``two row'' Springer fibers.
	However, this case is known to be far from generic.
	
	This paper deepens this connection with a similar correspondence in the substantially more complicated ``two column'' case.
	In particular, and building on works of Fresse, Melnikov, and Sakas-Obeid, we use webs to give a clean characterization of the smooth components of two column rectangle Springer fibers and a simple description of the geometry of these smooth components.
	We also show that the Poincar\'e polynomial of the smooth components is invariant under the natural dihedral action on the corresponding webs.
\end{abstract}

\maketitle

\vspace{-2em}
\section{Introduction} \label{intro}

The set of nilpotent elements in a semisimple Lie algebra $\mathfrak g$ forms a normal variety \cite{Kostant} that relates to the geometry of Schubert varieties \cite{moreKostant, evenmoreKostant, Carrell}, and is of significant current interest, e.g., \cite{MV,Nevins,CHY}.
Springer \cite{SpringerRes} constructed a resolution of this variety, whose fibers are now known as \emph{Springer fibers}.
In this paper, we consider the type $A$ case $\mathfrak g = \sl_n$.
Here, Springer fibers are subvarieties of the \emph{complete flag variety} ${\rm SL}_n(\mathbb C)/B$, where $B$ is the Borel subgroup of upper-triangular matrices, and are indexed by \emph{integer partitions} $\eta = (\eta_1, \eta_2, \ldots)$ of $n$.
In this case, Springer \cite{SpringerTrig} (see also \cite{Slodowy, Lusztig, BM, PrecupRichmond}) showed that the cohomology of the Springer fiber $S_\eta$ carries an action of the symmetric group $\Sn$ and, moreover, the top-graded piece of the cohomology is isomorphic to the \emph{Specht module} $\mathsf S^\eta$.
In other contexts, Springer fibers appear in the construction of \emph{convolution algebras} \cite{SW}, in knot homology \cite{KhovanovTangles, SeidSmith, CautKam1} (where the cohomology of two row rectangular Springer fibers relates to \emph{categorification of the Jones polynomial} \cite{Jones}), and in the combinatorics of \emph{Young tableaux} (for instance, \cite{Spaltenstein, vL, FMSO}).

Despite this long and rich history, the geometry and topology of Springer fibers are poorly understood in general.
Outside of some special cases, we only know that they are connected \cite[Chapter II.1]{SpaltBook} and the following.

\begin{thm}[{\cite{Spaltenstein} (see also \cite{Vargas, vL})}] \label{springer components}
	The Springer fiber $S_\eta$ is equidimensional of dimension $\sum_{i\ge 1} (i-1)\eta_i$.
	Moreover, the components of $S_\eta$ are in bijective correspondence with the \emph{standard Young tableaux} of shape $\eta$.
\end{thm}

The present manuscript treats the \emph{two column} case, where $\eta = (2, \ldots, 2, 1, \ldots, 1)$.
In this case, Fresse and Melnikov \cite{FMsingcomp} characterize the tableaux that correspond to smooth components of $S_\eta$ and, with Sakas-Obeid \cite{FMSO}, describe the geometry of each smooth component from its tableau.
Our main result is a diagrammatic reinterpretation of these two column results, using \emph{degree two $\sl_k$ webs}.

\begin{figure}[htbp]
	\begin{tikzpicture}
		\node at (-6.6, 0) {};
		\node at (-6.4, 0) {\includegraphics[scale=1]{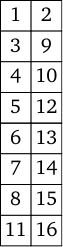}};
		\node at (-2.5, 0) {\includegraphics[scale=1.8]{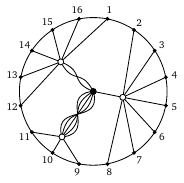}};
					
		\node at (2.5, 0) {\includegraphics[scale=1]{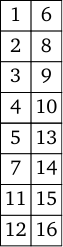}};
		\node at (6.4, 0) {\includegraphics[scale=1.8]{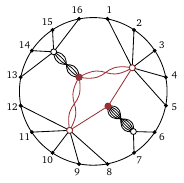}};
	\end{tikzpicture}

	\caption{Two column tableaux and their associated degree two $\sl_8$ webs.}
	\label{intro eg}
\end{figure}

Webs are planar graphs embedded in a disc and themselves have a rich history, originating with Kuperberg in 1996 \cite{Kuperberg} as diagrammatic tools for computation in the representations of Lie groups and their quantum deformations.
Webs have important applications to areas such as quantum topology \cite{KhovanovLinkInv, Higgins, Bodish, BodishWu}, cluster algebras \cite{FominP, FraserP}, dimer models \cite{Lam, FLL, DKS}, dynamical algebraic combinatorics \cite{PPR,HopRubey}, the geometric Satake correspondence \cite{FKK}, and skein algebras \cite{LS, IK, SSW}.
In the two row case, webs are well-understood: $\sl_2$ basis webs are \emph{noncrossing matchings}, sometimes called the \emph{Temperley--Lieb basis}; see \cite{RTW,TL,KungRota}.
For the present manuscript, we study the basis webs of \cite{Fraser2,GPPSS} for the \emph{two column} case, also known as the \emph{Pl\"ucker degree two $\sl_k$} case.
These basis webs are in bijection with rectangular two column standard Young tableaux.

Our main result is that two column basis webs characterize and describe the geometry of the smooth components of two column rectangle Springer fibers $S_{(2, 2, \ldots, 2)}$.
This characterization and description are purely graph-theoretic.
A \textbf{claw} is a subgraph consisting of an internal vertex at distance one from the boundary vertices, together with its adjacent boundary vertices.
The \textbf{size} of a claw is the number of its boundary vertices.

\begin{mainthm*}
	Let $W$ be the degree two $\sl_k$ web associated to a component $S_W$ of the two column rectangle Springer fiber $S_{(2, 2, \ldots, 2)}$.
	\begin{itemize}
		\item (\Cref{tree iff smooth})
		$S_W$ is smooth if and only if the underlying graph of $W$ is a forest.
		
		\item (\Cref{two col geometry})
		Suppose that $S_W$ is smooth and $W$ has $2k$ boundary vertices.
		If the underlying graph of $W$ is disconnected, then let $i$ be the maximum integer such that the vertices $1, 2, \ldots, i$ are in the same claw of $W$.
		Then, the component $S_W$ is an iterated fiber bundle with base
		\begin{equation*}
			\big( \Fl(i) \times \Fl(k) ,~ \mathbb P^i ,~ \mathbb P^{i+1} ,~ \ldots ,~ \mathbb P^{k-1} \big).
		\end{equation*}
		
		Otherwise, the interior of $W$ is connected and has three claws.
		Denote by $c_i$ the size of the $i$-th claw in clockwise order, and let $\ell$ be the multiplicity of the edge between the second claw and the filled internal vertex.
		The component $S_W$ is an iterated fiber bundle with base
		\begin{equation*}
			\big( \Fl(c_1) \times \Fl(c_2) ,~ \Gr_\ell(c_3) ,~ \mathbb P^{k-c_3} ,~ \mathbb P^{k-c_3+1} ,~ \ldots ,~ \mathbb P^{k-1} \big).
		\end{equation*}
	\end{itemize}
\end{mainthm*}

We demonstrate this theorem in \Cref{contrasting eg,move equiv comps,geo example}.
Our results are reinterpretations of works of Fresse, Melnikov, and Sakas-Obeid \cite{FMsingcomp, FMSO} that were given in terms of standard Young tableaux.
\Cref{contrasting eg} contrasts our two rules.
This perspective led to a correction \cite{CMansour} of an enumeration of Mansour \cite{Mansour}.

Experts have known since Fung's thesis \cite{FungThesis, Fung} that webs govern the geometry and topology of \emph{two row rectangle} Springer fibers $S_{(k, k)}$.
Indeed, the \emph{noncrossing matching diagrams}, in work of Fung and others \cite{Russell, SW}, coincide with $\sl_2$ webs. 
However the two row case is far from generic; for instance, each component is a smooth, iterated $\mathbb P^1$-bundle.
The same does not hold for arbitrary Springer fibers, which admit singular components \cite{FMsmoothchar}.
Although not smooth, the components of the two column Springer fibers, which we consider in this paper, are normal, Cohen--Macaulay, and have rational singularities \cite{PS}.

Since webs describe the geometry of two row and smooth components of two column Springer fibers, we record the following question, which has been asked by experts (e.g., \cite{TymoczkoTalk}) but has not yet appeared explicitly in the literature.

\begin{question}
	Do webs describe the geometry of the components of Springer fibers in general?
\end{question}

From their embedding into a disc, webs inherit a natural dihedral action, corresponding to rotation and reflection of the underlying graph.
Combinatorially, rotation of webs corresponds to dynamics called \emph{promotion} on the associated tableaux \cite{White,PPR,GPPSSprom}.
Also, webs correspond to vectors in certain invariant spaces of tensors of representations.
It is desirable to determine a \emph{rotation-invariant web basis} for this space: a set of webs that is closed under diagrammatic rotation (up to sign) and that corresponds to a basis of the invariant space.
Kuperberg \cite{Kuperberg} describes such a basis in the $\sl_2$ and $\sl_3$ cases, yet it took nearly 30 years to solve the degree two $\sl_k$ \cite{GPPSS} and $\sl_4$ \cite{GPPSS4} cases.
To complement the combinatorial and representation-theoretic interpretations of rotation, we give the following geometric interpretation.

\begin{corollary*}[{\Cref{dihedral invariance of smooth components}}]
	Two smooth components of $S_{(2, 2, \ldots, 2)}$ have the same Poincar\'e polynomial if and only if their webs are a rotation or reflection of one another.
\end{corollary*}

\begin{example} \label{contrasting eg}
	Consider the degree two $\sl_8$ webs and their corresponding tableaux given in \Cref{intro eg}.
	The first web is a tree, hence corresponds to a smooth component of its associated Springer fiber (\Cref{tree iff smooth}).
	(The disc in which the web is embedded does not contribute any edges to the graph.)
	Its first claw involves the boundary vertices $2, 3, \ldots, 8$.
	Hence, from \Cref{two col geometry}, the component corresponding to the first web is an iterated fiber bundle whose base is given by $\big( \Fl(7) \times \Fl(3), \Gr_5(6), \mathbb P^2, \mathbb P^3, \ldots, \mathbb P^7 \big)$.
	The second web contains a cycle, drawn in red, hence corresponds to a singular component of the Springer fiber.
	
	We contrast this with the procedure outlined in \cite{FMsingcomp,FMSO}, which we discuss in \Cref{two col spring}.
	First, to check for singularities, we note that both tableaux have exactly three entries $j$ in their first columns such that $j+1$ is in the second column.
	In the first tableau, when $i=1$, the $i$-th entry in the second column is exactly $2i$, so the corresponding component is smooth.
	However, for the second tableau, this does not occur for any $i < k$, where $k=8$ is the number of rows; hence this tableau corresponds to a singular component.
	To describe the geometry of the component corresponding to the first tableau, we set $a = 5$, $b = 2$, and $c = 1$ per \Cref{FMSO triple}, and use \Cref{FMSO geometry} to conclude that the corresponding base is $\big( \Fl(a+b) \times \Fl(b+c), \Gr_a(a+c), \mathbb P^b, \mathbb P^{b+1}, \ldots, \mathbb P^{k-1} \big)$, which agrees with our first calculation.
\end{example}

\subsection*{Outline of the paper}
We first discuss the background and relevant results from the literature of Springer fibers and webs in \Cref{Background}.
\Cref{Geometry} treats the case of two column rectangular Springer fibers, each subsection treating one of the three main results above: characterization of smooth components, description of the geometry of the smooth components, and the Poincar\'e polynomials.
We discuss in \Cref{TwoRowWebs} an extension of our results to the non-rectangular two column case using \emph{noncrossing matching and ray diagrams}.

\section*{Acknowledgements}
The author thanks his advisor, Oliver Pechenik, for invaluable suggestions and feedback that substantially improved this manuscript.
We thank Ronit Mansour for helpful discussions about \cite{Mansour}.
We are also grateful to Stephan Pfannerer and Joshua P. Swanson for helpful conversations about webs and for Sage code used to produce many of the figures.
We thank Joel Kamnitzer, Jake Levinson, and Martha Precup for helpful conversations---particularly those involving \Cref{iso or homeo components}---and Steven N. Karp for his comments on an earlier draft.
An anonymous referee was particularly helpful for improving the exposition in \Cref{deg 2 sl webs}.
Lastly, we acknowledge \cite{OEIS} for the discovery of the equinumerosity in \Cref{smooth pattern avoidance}.

The author is supported by NSERC Alexander Graham Bell CGS-D 588999-2024, a University of Waterloo President's Graduate Scholarship, and Oliver Pechenik's NSERC Discovery Grant RGPIN-2021-02391.

\section{Background} \label{Background}

In this section we review standard definitions and results from the literature relevant to the present manuscript.
First, we establish the necessary notation for integer partitions and tableaux in \Cref{tableaux}.
Then, in \Cref{spring}, we discuss Springer fibers and the correspondence with tableaux, and specialize in \Cref{two col spring} to the two column case.
\Cref{webs} discusses webs, and we again specialize to the two column case in \Cref{deg 2 sl webs}.

\subsection{Combinatorics of tableaux} \label{tableaux}

We briefly review some combinatorial notions, referring the reader to \cite{Sagan} for further details.
An \textbf{integer partition of $n \in \mathbb Z^+$} (or simply, a \textbf{partition of $n$}) is a weakly decreasing sequence $\eta = (\eta_1, \eta_2, \ldots)$ of positive integers such that $\sum_i \eta_i = n$.
The \textbf{Young diagram} $D_\eta$ of $\eta = (\eta_1, \eta_2, \ldots)$ is a left-justified grid of $n$ boxes such that the $i$-th row has exactly $\eta_i$-many boxes.
We draw our diagrams in \emph{English notation}, so they are indexed using matrix coordinates and the first row is at the top of the diagram.
The \textbf{conjugate} $\eta^*$ of $\eta$ is the partition obtained by taking the transpose of the diagram $D_\eta$, that is, swapping its rows and columns.
We often describe integer partitions by the shape of their diagram.
For instance, the diagram of the partition $\eta = (\eta_1, \eta_2)$ has two rows, so $\eta$ is a two row partition and $\eta^*$ is a two column partition, and when $\eta_1 = \eta_2$, we say that $\eta^*$ is a \textbf{two column rectangle}.

\begin{figure}[htbp]
	\begin{tikzpicture}
	
		\node at (0, 0) {\includegraphics[scale=1.1]{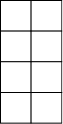}};
		\node at (2.5, 0) {\includegraphics[scale=1.1]{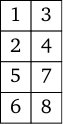}};
		\node at (5, 0) {\includegraphics[scale=1.1]{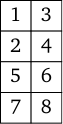}};
		\node at (7.5, 0) {\includegraphics[scale=1.1]{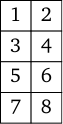}};
		
	\end{tikzpicture}
	\caption{The diagram of the two column rectangle $\eta = (2, 2, 2, 2) = (4, 4)^*$ and three standard Young tableaux of shape $\eta$.} 
	\label{diagram and tableaux}
\end{figure}

For any partition $\eta$ of $n$, a \textbf{filling} of its diagram $D_\eta$ is a bijective assignment of the values $[n] \coloneqq \{1, 2, \ldots, n \}$ to the boxes in $D_\eta$.
A \textbf{standard Young tableau} of shape $\eta$ is a filling of $D_\eta$ that is increasing along rows and columns from left-to-right and top-to-bottom, respectively.
We denote by $\syt(\eta)$ the set of all standard Young tableaux of shape $\eta$.
The number of standard Young tableaux of a given shape is given by the hook-length formula \cite[Theorem 7.3.1]{Sagan}.
\Cref{diagram and tableaux} illustrates the diagram $D_\eta$ and three standard Young tableaux of shape $\eta = (2, 2, 2, 2)$.

\subsection{Springer fibers} \label{spring}

Springer fibers were introduced in work of Springer \cite{SpringerRes,SpringerTrig}, arising from his study of the \emph{nilpotent cone}, the variety of nilpotent elements in a semisimple Lie algebra.
For an overview of the geometry and combinatorics of Springer fibers, we direct the reader to the survey paper of Tymoczko \cite{TymoczkoSurvey} and references therein.

The Jordan canonical form of a nilpotent operator on $\mathbb C^n$ corresponds uniquely to a partition of $n$, so we abuse notation and write $\eta$ both for a partition and its corresponding nilpotent operator.
We work in type $A$, where the \textbf{(complete) flag variety} $\Fl(n)$ in $\mathbb C^n$ is the set of nested $\mathbb C$-vector subspaces $V_\bullet = (V_0 \subset V_1 \subset \cdots \subset V_n)$ such that $\dim_{\mathbb C}V_i = i$ for all $i$.
Each element $V_\bullet$ of $\Fl(n)$ is called a \textbf{flag}.
The \textbf{Grassmannian} $\Gr_d(n)$ is the set of $d$-dimensional vector subspaces of $\mathbb C^n$.
\textbf{Projective space} $\mathbb P^n$ is the set of lines in $\mathbb C^{n+1}$, so $\mathbb P^n = \Gr_1(n+1)$ and $\dim_{\mathbb C} \mathbb P^n = n$.

The \textbf{Springer fiber} $S_\eta$ associated to a nilpotent operator $\eta: \mathbb C^n \to \mathbb C^n$ is the subvariety of $\Fl(n)$ consisting of the flags $V_\bullet$ such that $\eta\cdot V_i \subseteq V_{i-1}$ for all $i$.
For any invertible $g \in {\rm SL}_n(\mathbb C)$, there is an isomorphism $S_\eta \cong S_{g\eta g^{-1}}$, so we freely take $\eta$ to be in Jordan form.
This gives a bijective correspondence between partitions of $n$ and (isomorphism classes of) Springer fibers in $\Fl(n)$.
We describe the Springer fiber $S_\eta$ with the same adjectives associated to $\eta$, for instance, if $\eta = (\eta_1, \eta_2)^*$ is a two column partition, then we say that $S_\eta$ is a \textbf{two column Springer fiber}.

To conclude this subsection, we survey some relevant results on the geometry of Springer fibers.
In general, the Springer fiber $S_\eta$ is connected \cite[Chapter II.1]{SpaltBook} and its components are singular \cite{Vargas, FMsmoothchar}.
Moreover, $S_\eta$ is equidimensional of dimension $\sum_{i\ge 1} (i-1)\eta_i$ and there is a bijective correspondence between its irreducible components and the standard Young tableaux of shape $\eta$; see \Cref{springer components}.
For $T \in \syt(\eta)$, we write $S_T$ for the corresponding irreducible component of $S_\eta$.
The dimension formula for $S_\eta$ is obtained from the diagram of $\eta$ by placing the value $i-1$ in every box in the $i$-th row of $D_\eta$, and summing the values in all the boxes.
For instance, when $\eta = (3, 2, 1)$, we obtain $\tiny \begin{ytableau}
	0 & 0 & 0 \\ 
	1 & 1 \\ 
	2
\end{ytableau}$, so $S_{(3,2,1)}$ has dimension $4$.

When $\eta = (\eta_1, 1, \ldots, 1)$ is a \emph{hook shape}, Fung \cite{Fung} describes the topology of the components of $S_\eta$ and their pairwise intersections.
Fung \cite{Fung} and \cite{CautKam1, Wehrli, RTtworow, Russell, SW, ILW} describe the geometry and topology of components of \emph{two row} Springer fibers and the pairwise intersections of these components.
Geometrically, in this two row case, the components are smooth, iterated $\mathbb P^1$-bundles, and the number of components in this bundle is described by certain \emph{noncrossing matching and ray diagrams}, which are in correspondence with two row standard Young tableaux.
Karp and Precup \cite{KarpPrecup} characterize the components of a Springer fiber that are equal to a Richardson variety.
Springer fibers are paved by affines \cite{Spaltenstein} and, in the two row case, this paving is described explicitly by noncrossing matching diagrams \cite{GNST}.

Every component of $S_\eta$ is smooth only when $\eta$ is either a hook \cite{Vargas}, has two rows \cite{Fung}, or is of the form $\eta = (\eta_1, \eta_2, 1)$ or $\eta = (2, 2, 2)$ \cite{Fresse222, FMsmoothchar}.
When $\eta = (\eta_1, \eta_2)^*$ is a two column partition, the components of $S_\eta$ are normal, Cohen--Macaulay, and have rational singularities \cite{PS}.

\subsection{Geometry for two column rectangle Springer fibers} \label{two col spring}

We henceforth consider two column rectangle Springer fibers $S_\eta$, that is, where $\eta = (k, k)^*$ for some $k$.
In \Cref{TwoRowWebs}, we discuss extensions of our work to the arbitrary two column case, using \emph{noncrossing matching and ray diagrams}.
Although these diagrams do correspond to two column standard Young tableaux, they are not webs in the representation-theoretic sense.

For any two column Springer fiber $S_\eta$, Fresse and Melnikov \cite{FMsingcomp} characterize the standard Young tableaux corresponding (under \Cref{springer components}) to smooth components of $S_\eta$.
Then, in their work with Sakas-Obeid \cite{FMSO}, they describe the geometry of these smooth components.
We now summarize their results in the two column rectangle case.

Fix a two column rectangle partition $\eta = (k, k)^*$ and standard Young tableau $T \in \syt(\eta)$.
Denote by $\col_i(T)$ the set of entries in the $i$-th column of $T$.
We write
\begin{equation*}
	\col_1(T) = \{ a_1 < a_2 < \cdots < a_k \} 
	\quad\text{and}\quad 
	\col_2(T) = \{ b_1 < b_2 < \cdots < b_k \}.
\end{equation*}
The column lattice word of $T$ is the word whose $i$-th entry is the index of the column that contains the box labelled $i$.
Let $\Asc(T)$ be the set of ascents in the column lattice word of $T$, that is,
\begin{equation*}
	\Asc(T) = \{ j \in \col_1(T) \mid j+1 \in \col_2(T) \},
\end{equation*}
the set of entries $j$ in the first column of $T$ for which $j+1$ is in the second column.
The set $\Asc(T)$ is denoted by $\tau^*(T)$ in \cite{FMsingcomp,FMSO}.
From the set $\Asc(T)$, we have the following characterization of the smooth components of $S_\eta$.
Recall that for a tableau $T \in \syt(\eta)$, we write $S_T$ for its corresponding component in $S_\eta$ under the bijection in \Cref{springer components}.
The following also appears as \cite[Theorem 2]{FMSO}.

\begin{theorem}[{\cite[Theorem 1.2]{FMsingcomp}}] \label{FM smooth characterization}
	Let $\eta = (k, k)^*$ and $T \in \syt(\eta)$.
	Then, the corresponding irreducible component $S_T$ of $S_\eta$ is smooth if and only if $|\Asc(T)| \le 3$ and, when $|\Asc(T)| = 3$, there is some $i \in [k-1]$ such that $b_i = 2i$.
\end{theorem}

The condition that $b_i = 2i$ says that the tableau consists of the (vertical) concatenation of two subtableaux, one with entries $1, 2, \ldots, 2i$, and one with $2i+1, 2i+2, \ldots, n$.
That $i < k$ prohibits either subtableau from being empty.

\begin{example} \label{eg smooth via asc}
	Consider the standard Young tableaux in \Cref{diagram and tableaux}, say $T_1$, $T_2$, and $T_3$ from left-to-right.
	These tableaux correspond to irreducible components of the Springer fiber $S_{(2, 2, 2, 2)}$.
	We leave it to the reader to verify that $|\Asc(T_1)| = 2$, $|\Asc(T_2)| = 3$, $|\Asc(T_3)| = 4$, and, moreover, that $S_{T_1}$ and $S_{T_2}$ are smooth components of $S_\eta$ but that $S_{T_3}$ is singular.
\end{example}

We denote by $\syts(\eta)$ the set of tableaux $T$ in $\syt(\eta)$ whose corresponding component $S_T$ is smooth.
The geometry of a smooth component of a two column Springer fiber is an \emph{iterated fiber bundle}.
The following definition is as in \cite[\textsection 2.3]{FMSO}.

\begin{definition} \label{iterated fiber bundle}
	Let $X$ and $B_1, \ldots, B_m$ be algebraic varieties.
	We say that $X$ is an \textbf{iterated fiber bundle of base $(B_1)$} if $X$ and $B_1$ are isomorphic.
	We say that $X$ is an \textbf{iterated fiber bundle of base $(B_1, \ldots, B_m)$} if there is a fiber bundle $X \to B_m$ whose fiber is an iterated fiber bundle of base $(B_1, \ldots, B_{m-1})$.
\end{definition}

For two examples, $B_1 \times \cdots \times B_m$ is an iterated fiber bundle with base $(B_1, \ldots, B_m)$, and the flag variety $\Fl(n)$ is an iterated fiber bundle with base $(\mathbb P^1, \mathbb P^2, \ldots, \mathbb P^{n-1})$.

\begin{remark} \label{point in bundle}
	Notice that if in \Cref{iterated fiber bundle} some $B_i$ is a point, it follows that $X$ is isomorphic to an iterated fiber bundle with base $(B_1, \ldots, \hat B_i, \ldots, B_m)$.
	Consequently, we may freely omit any component of the base that is a point.
\end{remark}

In \cite{FMSO}, the geometry of smooth components of two column Springer fibers is described using the following data from the corresponding standard Young tableau.

\begin{definition}[{\cite[\textsection 2.4]{FMSO}}] \label{FMSO triple}
	Let $T \in \SYTs((k, k)^*)$.
	\begin{itemize}
		\item If $\Asc(T) = \{ \alpha \}$, then $a_T = \alpha - k$, $b_T = k$, and $c_T = k - \alpha$,

		\item If $\Asc(T) = \{ \alpha < \beta \}$, then $a_T = k - (\beta - \alpha)$, $b_T = \beta - k$, and $c_T = k - \alpha$,
		
		\item If $\Asc(T) = \{ \alpha < \beta < \gamma \}$, then $a_T = k - (\gamma - \beta)$, $b_T = (\gamma - \alpha) - k$, and $c_T = k - (\beta - \alpha)$.
	\end{itemize}
\end{definition}

The values $a_T, b_T, c_T$ are nonnegative and satisfy the relations $b_T \ge 1$ and $a_T + b_T + c_T = k$ \cite[\textsection 2.4]{FMSO}.

\begin{example}
	We continue \Cref{eg smooth via asc} and compute the values $a_T, b_T, c_T$ for the tableaux $T_1$ and $T_2$.
	For both tableaux we have $k = 4$.
	From $\Asc(T_1) = \{ 2, 6 \}$, we have that $a_{T_1} = 4 - (6-2) = 0$, $b_{T_1} = 6-4 = 2$, and $c_{T_1} = 4-2 = 2$.
	Similarly, from $\Asc(T_2) = \{ 2, 5, 7 \}$, we compute that $a_{T_2} = 4 - (7-5) = 2$, $b_{T_2} = (7-2) - 4 = 1$, and $c_{T_2} = 4 - (5-2) = 1$.
\end{example}

\begin{theorem}[{\cite[Theorem 3]{FMSO}}] \label{FMSO geometry}
	Let $T \in \syts((k, k)^*)$ and define $a_T$, $b_T$, and $c_T$ as in \Cref{FMSO triple}.
	Then, $S_T$ is an iterated fiber bundle with base
	\begin{equation*}
		\big( \Fl(a_T + b_T) \times \Fl(b_T + c_T) ,~ \Gr_{a_T}(a_T + c_T) ,~ \mathbb P^{b_T} ,~ \mathbb P^{b_T+1} , \ldots, \mathbb P^{k-1} \big).
	\end{equation*}	
\end{theorem}

Their proof \cite[Section 3]{FMSO} uses combinatorics of tableaux---\emph{jeu de taquin} and a deletion procedure they call \emph{projection}---to reduce to the case where the geometry can be understood from a description of the component similar to that of \cite[Proposition 2.2]{Vargas}.
We will see, in \Cref{two col geometry}, a reinterpretation of this result in terms of \emph{webs}.

As an application of our work, we consider the \emph{Poincar\'e polynomials} of the smooth components of two column rectangle Springer fibers.
For an algebraic variety $X$, denote by $H^i(X, \mathbb C)$ the $i$-th graded piece of its sheaf cohomology group.
The \textbf{Poincar\'e polynomial} of $X$ is 
\begin{equation*}
	P_X(q) = \sum_{i \ge 0} q^i \dim_{\mathbb C} H^{2i}(X, \mathbb C).
\end{equation*}
Since $H^i(X, \mathbb C) = 0$ for all $i > 2\dim_{\mathbb C} X$, only finitely-many terms of this sum are nonzero.
For the iterated fiber bundles $X$ with base $(B_1, \ldots, B_\ell)$ that we encounter in this paper, we have that $P_X(q) = P_{B_1}(q)P_{B_2}(q) \cdots P_{B_\ell}(q)$.
(See \cite[Chapter 21]{BT} for more details.)

In view of \Cref{FMSO geometry}, the Poincar\'e polynomials of the components of our Springer fibers are products of the Poincar\'e polynomials of projective spaces, flag varieties, and Grassmannians.
These Poincar\'e polynomials are given in terms of \emph{$q$-integers}.
For any $n \in \mathbb N$, define the \textbf{$q$-integer} $[n]_q = 1 + q + q^2 + \cdots + q^{n-1}$.
Notice that when $q=1$ we have $[n]_{q=1} = n$.
Using this, we define \textbf{$q$-factorials} and \textbf{$q$-binomials} in analogy with their usual definitions.
That is, we take $[0]_q! = 1$ and $[n]_q! = [n]_q[n-1]_q!$ for any $n \ge 1$.
For any $d \le n$,
\begin{equation*}
	\begin{bmatrix} n \\d \end{bmatrix}_q = \frac{[n]_q!}{[d]_q! ~ [n-d]_q!}.
\end{equation*}
When $d > n$, the $q$-binomial is $0$.
If $d = 0$, then the $q$-binomial is 1 for all $n \ge 0$.
Arguing inductively and using a $q$-analogue of the binomial identity ${n\choose d} = {n-1\choose d} + {n-1\choose d-1}$, one can show that the $q$-binomial is a polynomial in $q$ \cite[Theorem 3.2.3]{Sagan}.
We henceforth suppress the subscript $q$ when there is no ambiguity.

\begin{fact} \label{basic poincare polys}
	The Poincar\'e polynomials of $\mathbb P^n$, $\Fl(n)$, and $\Gr_d(n)$ are, respectively,
	\begin{equation*}
		P_{\mathbb P^n}(q) = [n+1], \quad
		P_{\Fl(n)}(q) = [n]!, \quad\text{and}\quad
		P_{\Gr_d(n)}(q) = \begin{bmatrix} n \\ d \end{bmatrix}.
	\end{equation*}
\end{fact}

The Poincar\'e polynomials for $\mathbb P^n$ and $\Fl(n)$ follow from that of $\Gr_d(n)$.
Indeed, since $\mathbb P^n = \Gr_1(n+1)$, we have $P_{\mathbb P^n}(q) = {\footnotesize \begin{bmatrix} n+1 \\ 1 \end{bmatrix}} = [n+1]$.
Also, since $\Fl(n)$ is an iterated fiber bundle with base $(\mathbb P^1, \mathbb P^2, \ldots, \mathbb P^{n-1})$, we have $P_{\Fl(n)}(q) = P_{\mathbb P^1}(q) P_{\mathbb P^2}(q) \cdots P_{\mathbb P^{n-1}}(q) = [n]!$.

\begin{theorem}[{\cite[Theorem 4]{FMSO}}] \label{invariant poincare}
	Let $\eta = (k, k)^*$ be a partition and $T, T' \in \syts(\eta)$ be tableaux corresponding to smooth components of $S_\eta$.
	Define $(a_T, b_T, c_T)$ and $(a_{T'}, b_{T'}, c_{T'})$ as in \Cref{FMSO triple} for $T$ and $T'$, respectively.
	Then, $P_{S_T}(q) = P_{S_{T'}}(q)$ if and only if either
	\begin{itemize}
		\item $0 \in \{ a_T, c_T \}$ and $0 \in \{ a_{T'}, c_{T'} \}$, or,
		\item $0 \notin \{ a_T, a_{T'}, c_T, c_{T'}\}$ and $\{ a_T, b_T, c_T \} = \{ a_{T'} , b_{T'}, c_{T'} \}$.
	\end{itemize}
\end{theorem}

In \Cref{dihedral invariance of smooth components}, we will see that this condition is equivalent to when the corresponding webs are in the same dihedral orbit.

\subsection{Webs} \label{webs}

Webs originated in work of Kuperberg \cite{Kuperberg} as a diagrammatic computation tool in the study of invariants of tensor representations of $\sl_2$, $\sl_3$, and their quantum deformations.
The $\sl_2$ webs coincide with \emph{noncrossing matchings} (which we will define in \Cref{deg 2 sl webs}) and are sometimes referred to as \emph{Temperley-Lieb webs}.
These webs are well-understood, and they can be traced back at least to work of Rumer, Teller, and Weyl; see \cite{RTW, TL, KungRota}.
For the purpose of diagrammatic computations, it is desirable to obtain \textbf{rotation-invariant web bases}: a set of webs that is closed under diagrammatic rotation (up to sign) and that corresponds to a basis in the tensor invariant space.
It is an open problem to find rotation-invariant web bases for general $\sl_k$; the $\sl_k$ \emph{degree two} \cite{GPPSS} and $\sl_4$ \cite{GPPSS4} cases were only recently solved.

Webs in this manuscript follow the $k$-valent, bicoloured by filled and unfilled vertices, and hourglass conventions of \cite{GPPSS,GPPSS4}.
That is, our webs have multiedges which are drawn as \textbf{hourglasses}, so that
between any two vertices, the order of the edges is the same when read clockwise around either vertex.
Moreover, they are \textbf{plabic}, as defined by Postnikov \cite{Postnikov}, meaning they have a \textbf{pla}nar embedding and \textbf{bic}olouring into filled and unfilled vertices.
We consider webs up to planar isotopy fixing the boundary disc.
The vertices are partitioned into \textbf{boundary} and \textbf{internal} vertices.
The boundary vertices are filled, labelled $1, 2, \ldots, 2k$, and lie on the boundary disc.
Although we draw the web with the embedding of underlying graph in a disc, this disc does not contribute any edges to the graph.
The \textbf{degree} of a vertex is the number of edges incident to it, counting with multiplicity.
Note that boundary vertices have degree $1$.
A web is \textbf{$k$-valent} (or \textbf{$k$-regular}) if each \emph{internal} vertex of the underlying graph has degree $k$.
Recall that a \textbf{claw} is a subgraph consisting of an internal vertex at distance one from the boundary vertices, together with its adjacent boundary vertices, and the \textbf{size} of a claw is the number of its boundary vertices.

\subsection{Degree two $\sl_k$ webs} \label{deg 2 sl webs}

In this paper, we consider the \emph{degree two $\sl_k$ basis webs} of \cite{Fraser2,GPPSS}.
The \emph{degree two} prefix arises from the fact that the corresponding tensor invariant space can be identified with a homogeneous piece of the coordinate ring of the Grassmannian spanned by products of \emph{pairs} of Pl\"ucker coordinates with disjoint support.
Certain equivalence classes of degree two $\sl_k$ webs are in bijection with two column rectangular standard Young tableaux \cite[Proposition 2.14]{Fraser2}, hence correspond to the irreducible components of the Springer fiber of the same partition shape by \Cref{springer components}.
Throughout we assume that our webs are \emph{fully reduced}, that is, every component of the web contains a boundary vertex.

Denote by $\mathcal B^{2k}$ the set of (fully reduced) degree two $\sl_k$ webs.
Diagrammatically, this is the set of $k$-valent hourglass plabic graphs with $2k$ boundary vertices, as defined in \Cref{webs}.
We consider $\mathcal B^{2k}$ up to an equivalence relation called \emph{moves}, described in \cite{GPPSS} (c.f., \cite[Proposition 2.14]{Fraser2}).
Our results are well-defined with respect to this equivalence relation, in the sense that any two move-equivalent webs have the same structure needed for our results.
See, for instance, \Cref{move equiv comps,first claw}.

We now describe a bijection from standard Young tableaux of shape $(k, k)^*$ to $\mathcal B^{2k}$ when $k \ge 2$.
Our presentation follows that of \cite{GPPSS}, which itself is an adaptation of work of Fraser \cite{Fraser2}.
Despite the map (possibly) involving a choice, the resulting equivalence class in $\mathcal B^{2k}$ is well-defined \cite[Proposition 2.14]{Fraser2}.
The bijection is the following composition of maps
\begin{alignat*}{2}
	\syt((k, k)^*)
	&\to \big\{ \text{noncrossing (perfect) matchings of $2k$ vertices} \big\} \\
	&\to \bigcup_{s=2}^k \big\{ \text{weighted dissections of an $s$-gon} \big\} \\
	&\to \bigcup_{s=2}^k \big\{ \text{weighted triangulations of an $s$-gon} \big\} 
	\to \mathcal B^{2k}.
\end{alignat*}

Let $G = (V, E)$ be a graph with vertices $V$ and edges $E$.
A \textbf{matching} of $G$ is a subset of edges $M \subseteq E$ such that every vertex in the subgraph $(V, M)$ has degree at most $1$.
When every vertex in $(V, M)$ has degree exactly $1$, we say that $M$ is a \textbf{perfect matching}.
We embed the vertices $1, 2, \ldots, 2k$ of $(V, M)$ sequentially clockwise around the boundary of a disc, and say that the matching $M$ is \textbf{noncrossing} if the edges can be drawn inside the disc without intersection.
As an abuse of terminology, we write noncrossing matching to mean noncrossing perfect matching.

We now describe each of the above maps in turn.
The first map is the usual bijection between two column rectangle standard Young tableaux and noncrossing matchings.
That is, it is the unique map that sends $T \in \syt((k, k)^*)$ to a noncrossing matching of $2k$ vertices such that each edge $\{ i, j \}$ with $i < j$ has $i$ in the first column of $T$ and $j$ in the second column.
Concretely, write the entries in the second column of $T$ as $\col_2(T) = \{ b_1 < b_2 < \cdots < b_k \}$.
We match the vertices $b_1, b_2, \ldots, b_k$ in this order, and match $b_i$ with the maximum unmatched entry in $\col_1(T)$ that is less than $b_i$.
\Cref{noncrossing construction} gives an example of such a matching built step-by-step.

\begin{figure}[htbp]
	\begin{tikzpicture}
		\node at (0, 0) {\includegraphics[scale=1]{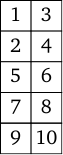}};
		\node at (2.5, 0) {\includegraphics[scale=1.4]{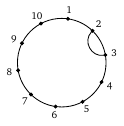}};
		\node at (5.5, 0) {\includegraphics[scale=1.4]{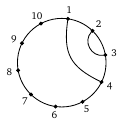}};
		\node at (8.5, 0) {\includegraphics[scale=1.4]{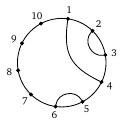}};
		\node at (11.5, 0) {\includegraphics[scale=1.4]{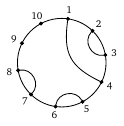}};
		\node at (14.5, 0) {\includegraphics[scale=1.4]{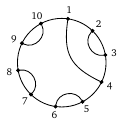}};
	\end{tikzpicture}

	\caption{Step-by-step construction of a noncrossing matching from the corresponding two column rectangular tableau.} 
	\label{noncrossing construction}
\end{figure}

Given a noncrossing matching on $2k$ vertices with $k \ge 2$, let $i_1 < \cdots < i_s$ be the subset of boundary vertices $\{1, 2, \ldots, 2k\}$ such that $\{ i_j, i_j+1 \} \mod 2k$ is an edge in the matching.
Each $i_j$ can be thought of as the first endpoint of an arc between consecutive vertices, when reading clockwise around the boundary.
Notice that $2 \le s \le k$.
Construct an $s$-gon by identifying the vertices $i_j+1, \ldots, i_{j+1}$ for all $j$, working mod $2k$.
Denote by $V_j$ the corresponding vertex, say, $V_j = \{ i_j+1, \ldots, i_{j+1} \}$, and draw each $V_j$ unfilled.
Between $V_j$ and $V_\ell$ introduce an edge with weight $m$, where $m$ is the number of edges between vertices of $V_j$ and $V_\ell$ in the noncrossing matching.
If $m = 0$, then there is no edge between $V_j$ and $V_\ell$.

For the example in \Cref{noncrossing construction}, 
we have that $i_1 = 2$, $i_2 = 5$, $i_3 = 7$, and $i_4 = 9$.
We thus obtain four vertices $V_1$, $V_2$, $V_3$, and $V_4$ by identifying, respectively, $\{ 3, 4, 5 \}$, $\{ 6, 7 \}$, $\{ 8, 9 \}$ and $\{ 10, 1, 2 \}$.
\Cref{weighted dissection} shows the resulting dissection with edge weights.

Our next step is to triangulate the $s$-gon by including nonintersecting diagonals, represented as weight 0 edges.
This may involve making a choice, but any choice gives a web in the same \emph{move equivalence class} in $\mathcal B^{2k}$; see our discussion at the beginning of this subsection.
In our example, we are required to introduce one edge of weight zero; we choose the edge between $\{8, 9\}$ and $\{3, 4, 5\}$.
\Cref{running eg webs} shows the difference in the resulting webs between choosing this edge and choosing the edge between $\{ 10, 1, 2 \}$ and $\{ 6, 7 \}$.

The penultimate step takes this weighted triangulation and produces the interior of our desired web.
In each of the triangles in the triangulation, place a filled vertex.
Add edges from this filled vertex to the three vertices of its bounding triangle.
The weight of a new edge $e$ in the interior of the web is the sum of weights in the triangulation of the edges \emph{opposite} $e$ in the following sense:
Each new edge $e$ has as one vertex a vertex of the $s$-gon; the other is a filled vertex $v$.
The edge $e$ and the other two edges incident to this filled vertex subdivide the ambient face of the triangulation into three triangles.
One of these three triangles is not incident to $e$, and there is an edge $f$ of this triangle that is not incident to $v$.
If $f$ is an edge of the $s$-gon, then the weight of $e$ is set to the weight of $f$.
Otherwise the edge $f$ bisects the triangulation.
In this case, the weight of $e$ is set to the sum of the weights of all edges, including $f$, in the bisected component that does not include $e$.

\Cref{weighted dissection} demonstrates this procedure with our running example.
We include here the computations of two edge weights.
Let $e$ be the edge between $\{8, 9\}$ and the filled vertex in the north-east face of the triangulation.
The opposite edge in this case is the boundary edge of the $4$-gon between the vertices $\{10, 1, 2\}$ and $\{3, 4, 5\}$.
This edge has weight $2$, so we give $e$ the same weight.
On the other hand, consider the edge $e$ incident to the boundary vertex $\{10, 1, 2\}$.
The opposite edge in this case is the zero-weight edge.
The edges in the component bisected by this edge are edges of the triangle between $\{8, 9\}$, $\{3, 4, 5\}$, and $\{6, 7\}$.
The sum of these edge weights is $0 + 1 + 1 = 2$.

\begin{figure}[htbp]
	\begin{tikzpicture}
		\node at (0, 0) {\includegraphics{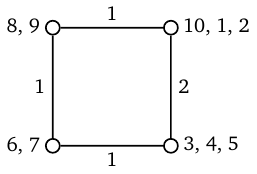}};
		\node at (5.5, 0) {\includegraphics{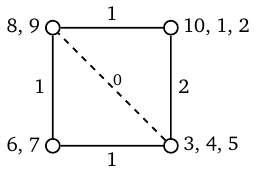}};
		\node at (11, 0) {\includegraphics{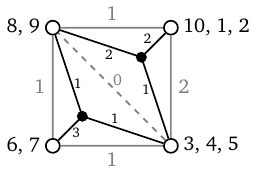}};
	\end{tikzpicture}

	\caption{Our running example, continued from \Cref{noncrossing construction}.
	Left-to-right: The weighted dissection, a weighted triangulation, and the interior of the resulting web overlaid on the weighted triangulation in \textcolor{gray}{grey}.}
	\label{weighted dissection}
\end{figure}

We now obtain our web as follows.
First, delete all edges that were present in the triangulation, leaving the vertices of the $s$-gon and the edges introduced in the previous step.
Replace the vertices of the $s$-gon as follows.
Take, in order, the unfilled vertices $V_j = \{ i_j+1 , \ldots, i_{j+1} \}$ of the $s$-gon.
Place $V_j$ inside the embedding disc, introduce a filled boundary vertex for each of $i_j+1, \ldots, i_{j+1}$and add edges between $V_j$ and each of these corresponding boundary vertices.
Place these boundary vertices sequentially clockwise on the disc (which, recall, does not contribute any edges to the graph).
Notice that each $V_j$ and its adjacent boundary vertices form a claw of the web.
In particular, the definition of the vertices $V_j$ guarantees that the number of claws is equal to the number of arcs in the noncrossing matching that are between consecutive vertices.
The interior of the web is exactly what was constructed in the previous step.
Lastly, replace each internal edge of weight $m$ with an $m$-hourglass edge.
An hourglass is a multiedge where the edges are twisted such that the clockwise order of edges is the same between the two endpoint vertices.
\Cref{running eg webs} shows the web for our running example.

\begin{figure}[htbp]
	\begin{tikzpicture}
		\node at (0, 0) {\includegraphics[scale=1.65]{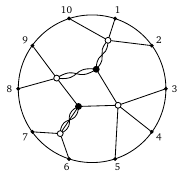}};
		\node at (6, 0) {\includegraphics[scale=1.65]{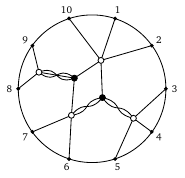}};
	\end{tikzpicture}
	
	\caption{The example from \Cref{noncrossing construction} continued.
	Left: The web obtained from the triangulation in \Cref{weighted dissection}.
	Right: The web obtained by choosing the other diagonal in the triangulation.
	These webs are related by a \emph{square move}; see \cite[Figure 3]{GPPSS}.}
	\label{running eg webs}
\end{figure}

We make the following notational convention.

\begin{definition} \label{web notation}
	Let $\eta = (k, k)^*$ and $T \in \syt(\eta)$.
	We write $M_T$ for the noncrossing matching associated to $T$, and $W_T$ for the degree two $\sl_k$ web associated to $T$.
\end{definition}

Through the above construction, we have shown the following.
Recall from \Cref{two col spring} that for a standard Young tableau $T$, we denote by $\col_i(T)$ the entries in column $i$ of $T$.
Also, when $T$ has two columns, then $\Asc(T)$ is the subset of entries $j$ in $\col_1(T)$ with $j+1 \in \col_2(T)$.
We also write $\col_2(T) = \{ b_1 < b_2 < \cdots < b_k \}$ and $[n] = \{1, 2, \ldots, n\}$.

\begin{lemma} \label{counting arcs i--i+1}
	Let $T \in \syt((k, k)^*)$ with corresponding noncrossing matching $M_T$ and degree two $\sl_k$ web $W_T$.
	The following are equal:
	\begin{enumerate}[leftmargin=2em, label=(\roman*)]
		\item \label{count tau} $|\Asc(T)|$, if there is some $i \in [k-1]$ with $b_i = 2i$, and $|\Asc(T)|+1$ otherwise,
		\item \label{count adj matching arcs} The number of arcs	 of the form $\{i, i+1 \} \mod 2k$ in $M_T$,
		\item \label{count claws} The number of claws in $W_T$.
	\end{enumerate}
\end{lemma}

\begin{proof}
	The equality between \ref{count adj matching arcs} and \ref{count claws} is immediate from the above construction: 
	In the notation of the construction, there is exactly one claw for each $V_j$, and there is exactly one $V_j$ for each edge between adjacent vertices in the matching.
	
	We show equality between \ref{count tau} and \ref{count adj matching arcs}.
	Given $T \in \syt((k, k)^*)$, we first notice that there is some $i \in [k-1]$ with $b_i = 2i$ if and only if the first $i$ rows of $T$ form a standard Young tableau of shape $(i, i)^*$.
	Indeed, there are $2i$ entries in the first $i$ rows of $T$ and, if $b_i = 2i$, then because the rows and columns of $T$ are strictly increasing, the entries in the first $i$ rows of $T$ are exactly $\{1, 2, \ldots, 2i\}$.
	Let $T'$ be the tableau consisting of these first $i$ rows of $T$.
	Then, the corresponding noncrossing matching $M_{T'}$ is a subgraph of $M_T$, since the construction of $M_T$ matches vertices in the second column in increasing order.
	In particular, in both $M_{T'}$ and $M_T$, the vertex $1$ is matched with some $2, 3, \ldots, 2i$, so $\{ 1, 2k \} = \{ 2k, 2k+1 \} \mod 2k$ is not an edge in $M_T$.
	So we conclude that in this case, there are exactly $|\Asc(T)|$ edges in $M_T$ of the form $\{ i, i+1\} \mod 2k$.
	
	If there is no $i \in [k-1]$ with $b_i = 2i$, then each $b_i$ with $i < k$ is matched in $M_T$ with a vertex $j > 1$.
	Hence we obtain in $M_T$ the edge $\{ 2k, 1 \} = \{ 2k, 2k+1 \} \mod 2k$, and, consequently, there are exactly $|\Asc(T)|+1$ edges in $M_T$ of the form $\{ i, i+1\} \mod 2k$.
\end{proof}

\section{Geometry of smooth components} \label{Geometry}

In this section, we prove our main results: characterizing and describing the geometry of smooth components of two column rectangle Springer fibers in terms of the webs described in \Cref{deg 2 sl webs}.
We give the characterization in \Cref{characterization} and the geometric description in \Cref{bundled geo}.
Then, in \Cref{Poincare}, we discuss the relation between the Poincar\'e polynomials of smooth components and rotation- and reflection-invariance of degree two $\sl_k$ webs.
In this section, we take \textbf{web} to mean degree two $\sl_k$ web.

\subsection{Characterization of smooth components} \label{characterization}

Fix a two column rectangle partition $\eta = (k, k)^*$ for some $k \ge 2$, and set $n = 2k$.
Recall from \Cref{spring} that we write $S_\eta$ for the Springer fiber associated to $\eta$ and $S_T$ for the component of $S_\eta$ corresponding to $T \in \syt(\eta)$.
Also, we write $W_T$ for the degree two $\sl_k$ web corresponding to $T$.
In this subsection, we reinterpret \Cref{FM smooth characterization} in terms of degree two $\sl_k$ webs; see \Cref{tree iff smooth}.

A graph is a \textbf{forest} if the underlying simple graph has no cycles, and is a \textbf{tree} if it is a forest that is also connected: there is a path between any two vertices.
We say that a web is a forest when its underlying graph is a forest, and similarly for trees.
In \Cref{intro eg}, the web on the left is a tree, and the web on the right is not a forest.
The web in \Cref{forest not tree} is a forest but not a tree.

\begin{figure}[htbp]
	\includegraphics[scale=1.65]{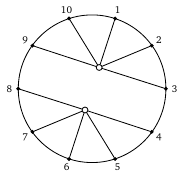}
	\caption{A degree two $\sl_5$ web that is a forest but not a tree.}
	\label{forest not tree}
\end{figure}

\begin{lemma} \label{forest via claw size}
	Let $T \in \SYT((k, k)^*)$.
	Then, $W_T$ is a forest if and only if it has at most $3$ claws.
\end{lemma}

\begin{proof}
	We leverage the construction in \Cref{deg 2 sl webs}.
	That is, consider the weighted triangulation of the polygon $P$ constructed from the noncrossing matching $M_T$.
	Because $k \ge 2$, the polygon $P$ has at least two vertices.
	The number of filled internal vertices in $W_T$ is exactly the number of faces in a triangulation of $P$.
	It follows from \Cref{counting arcs i--i+1} that the number of vertices of $P$ is exactly the number of edges of the form $\{ i, i+1 \} \mod 2k$ in the noncrossing matching $M_T$.
	So we claim that $W_T$ is a forest if and only if $P$ is a line or a triangle.
	
	The converse is clear, so suppose that the polygon $P$ has at least four vertices and consider any triangulation of $P$.
	Take any two vertices $u, v$ of $P$ that are adjacent by an internal edge $e$ in the triangulation, that is, an edge that is not on the boundary of the polygon in the weighted dissection.
	Then, both $u$ and $v$ are incident to the two distinct faces incident to $e$, say $f, g$.
	Consequently, the vertices in $W_T$ corresponding to $u$ and $v$ are both adjacent to the filled internal vertices corresponding to $f$ and $g$, which gives a 4-cycle in $W_T$.
\end{proof}

We have the following graph-theoretic reinterpretation of \cite[Theorem 1.2]{FMsingcomp}.

\begin{theorem} \label{tree iff smooth}
	Let $\eta = (k, k)^*$ and $T \in \SYT(\eta)$.
	Consider the component $S_T$ of $S_\eta$ associated to $T$ and the degree two $\sl_k$ web $W_T$ associated to $T$.
	Then, $S_T$ is smooth if and only if $W_T$ is a forest.	
\end{theorem}

\begin{proof}
	In our setting, \Cref{FM smooth characterization} says the component of $S_\eta$ associated to $T$ is smooth if and only if $|\Asc(T)| \le 3$ and, when equality holds, we moreover have that there is some $i \in [k-1]$ such that $b_i = 2i$.
	This occurs, by \Cref{counting arcs i--i+1}, if and only if $M_T$ has at most three arcs of the form $\{ i, i+1 \} \mod 2k$.
	\Cref{forest via claw size} guarantees that this is equivalent to $W_T$ being a forest.
\end{proof}

\begin{example} \label{move equiv comps}
	Consider the webs in \Cref{intro eg}.
	The web on the left is a forest, hence corresponds to a smooth component of the associated Springer fiber.
	In contrast, the web on the right is not a forest, so its corresponding component is singular.
	The webs in \Cref{running eg webs} are in the same \emph{move equivalence class} (as in \cite{GPPSS}), so correspond to the same component of the associated Springer fiber, and this component is singular.
	All of the webs in \Cref{forest not tree,webs for 222} correspond to smooth components.
\end{example}

As a consequence of \Cref{tree iff smooth}, we give the following enumeration of the number of smooth components of the two column rectangle Springer fiber $S_{(k, k)^*}$.
This result brought to our attention an error in an enumeration of Mansour \cite{Mansour}, which we have corrected \cite{CMansour}.
In \Cref{webs for 222}, we illustrate the 5 degree two $\sl_3$ webs, all of which correspond to smooth components of $S_{(2, 2, 2)}$.

\begin{figure}[htbp]
	\begin{tikzpicture}
		\node at (0, 0) 	{\includegraphics[scale=1]{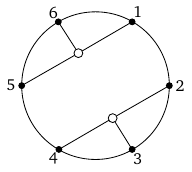}};
		\node at (3.25, 0) 	{\includegraphics[scale=1]{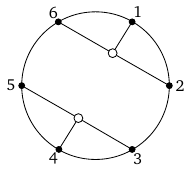}};
		\node at (6.5, 0) 	{\includegraphics[scale=1]{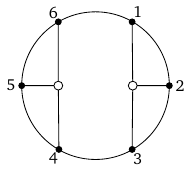}};
		\node at (9.75, 0) 	{\includegraphics[scale=1]{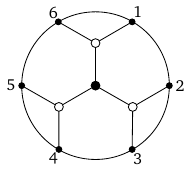}};
		\node at (13, 0) 	{\includegraphics[scale=1]{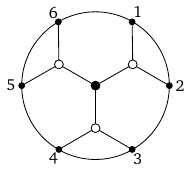}};
	\end{tikzpicture}
	
	\caption{The 5 degree two $\sl_3$ webs.}
	\label{webs for 222}
\end{figure}

\begin{corollary} \label{counting smooths}
	The number of smooth components of $S_{(k, k)^*}$ is $k + 2{k \choose 3}$, for any $k \ge 2$.
\end{corollary}

We give two proofs of \Cref{counting smooths}.
Both proofs make use of the fact that the reduced two column forest webs in this manuscript use are the unique elements in their square move equivalence classes.
This can be seen in the construction given in \Cref{deg 2 sl webs}, where the only triangulation is trivial because the $s$-gons for smooth components have $s=2$ in the disconnected case and $s=3$ in the connected case.

\begin{proof}[First proof of \Cref{counting smooths}]
	Let $W$ be a degree two $\sl_k$ web with boundary vertices $1, 2, \ldots, 2k$.
	We call the arc of the boundary of the disc between vertices $j$ and $j+1$ a \textbf{break} if $j$ and $j+1$ are in different claws of $W$.
	Notice that $W$ is determined uniquely by its breaks.
	Indeed, since $W$ has no cycles, it is the unique element in its move equivalence class.
	If $W$ has exactly two claws, then the breaks occur at $(j, j+1)$ and $(j+k, j+k+1)$, because the internal vertices of $W$ have degree $k$.
	Hence there are $k$ webs with exactly two claws.
	
	Consider now webs with exactly three claws.
	Each unfilled internal vertex incident to a claw has degree $k$ and is adjacent to the filled internal vertex, hence is adjacent to at most $k-1$ boundary vertices.
	Consequently, the distance between consecutive breaks is at most $k-1$.
	We construct two webs with exactly three claws by picking three breaks among the arcs of the boundary of the disc $(1, 2)$, $(2, 3)$, \ldots, $(k, k+1)$.
	There are ${k \choose 3}$ ways to make this selection and each selection gives rise to exactly two webs with three claws: one by moving the middle break antipodally, and one by moving the first and third breaks antipodally.
	The desired formula follows.
	Our observation that the distance between consecutive breaks is at most $k-1$ guarantees that we obtain every three-claw web in this way.
\end{proof}

\begin{proof}[Second proof of \Cref{counting smooths}]
	There are $k$ degree two $\sl_k$ webs with exactly two claws.
	Indeed, for each $i \in [k]$, the claws in the corresponding web have vertices $\{i, i+1, \ldots, i+k-1 \}$ and $\{ i+k, \ldots, 2k, 1, \ldots, i-1 \}$.
	
	So it remains to show that the number of webs with exactly three claws is exactly $2{k \choose 3}$.
	We identify a three-claw web with its three breaks between claws, as defined in the first proof of \Cref{counting smooths}.
	Recall that, because an $\sl_k$ web is $k$-valent, the distance between consecutive breaks is at least $2$ and at most $k-1$.

	For the first break, there are $2k$ choices.
	The second break can be placed anywhere with distance in $\{ 2, 3, \ldots, k-1 \}$ from the first break.
	There are two choices for the second break at each of these distances.
	If the second break is at distance $i$ from the first, then there are $i-1$ choices for the remaining break to be within distance $k-1$ of both the first and second breaks.
	So, the number of three-claw webs is 
	\begin{equation*}
		\frac1{3!} \left( 2k\sum_{i=2}^{k-1} 2(i-1) \right)
		= \frac2{3!} \left( 2k \sum_{i=1}^{k-2} i \right)
		= \frac2{3!} \left( 2k \cdot \frac{(k-1)(k-2)}{2} \right)
		= 2{k \choose 3}.
		\qedhere
	\end{equation*}
\end{proof}

In Schubert calculus, singularities and other properties are often characterized in terms of \emph{pattern avoidance} conditions.
Consider permutations $v \in \mathfrak S_m$ and $w \in \mathfrak S_k$ with $m \le k$.
Write $v$ and $w$ in \emph{one-line notation} $v = v_1v_2 \cdots v_m$ and $w = w_1w_2\cdots w_k$, so that $v(j) = v_j$ and $w(i) = w_i$.
We say that $w$ \textbf{contains} the pattern $v$ if there is a subsequence $w_{i_1}w_{i_2} \cdots w_{i_m}$ of $w$ with $i_1 < i_2 < \cdots < i_m$ such that $w_{i_j} < w_{i_\ell}$ if and only if $v_j < v_\ell$ for all $j, \ell$.
When this occurs, the entries in the subsequence $w_{i_1}w_{i_2}\cdots w_{i_m}$ have the same relative order as those in $v$.
If $w$ does not contain $v$, then we say that $w$ \textbf{avoids} the pattern $v$, or that $w$ is $v$-avoiding.
For instance, the permutation \underline{3}5\underline{26}1\underline{4} contains the pattern 2143, while the permutation 356214 avoids 2143.

We have the following.

\begin{corollary} \label{smooth pattern avoidance}
	Let $\eta = (k, k)^*$ with $k \ge 3$.
	The smooth components of $S_\eta$ are equinumerous with the permutations in $\mathfrak S_{k}$ that avoid all of the patterns $321$, $2143$, and $3124$.
\end{corollary}

\begin{proof}
	Since 
	\begin{equation*}
		k + 2{k\choose 3}
		= k + \frac{k(k-1)(k-2)}{3}
		= \frac{k^3 - 3k^2 + 5k}{3},
	\end{equation*}
	the result follows from combining \Cref{counting smooths} and \cite[\href{https://oeis.org/A116731}{A116731}]{OEIS}.
\end{proof}

It would be interesting to explicitly describe the correspondence between these smooth components and pattern-avoiding permutations.
Permutations that avoid $321$ are called \emph{fully commutative} and appear throughout algebraic combinatorics; see e.g., \cite{BJS, Stembridge, GreenKL, BHY}.
The permutations avoiding $2143$ are called \emph{vexillary} and also appear in both algebraic combinatorics and Schubert calculus \cite{LascouxSchutz, KMY, Anderson}.
For instance, the \emph{Schubert variety} $X_w = \overline{B_-wB/B}$ is smooth if and only if $w$ avoids both $2143$ and $1324$ \cite[Theorem 2.2]{LakSand}.
(Here, $B$ and $B_-$ are the subgroups of invertible upper-triangular and, respectively, lower-triangular matrices; see, for instance, \cite[p. 206]{WY} for this smoothness statement in the conventions of the present manuscript.)
We refer the reader to \cite{AbeBilley} for more on properties characterized by pattern avoidance.

After this manuscript first appeared online as a preprint, Striker and Tenner \cite{StrikerTenner} found a bijection between forest two column webs and a different class of permutations: those that avoid all of the patterns $132$, $4321$, and $3214$.

\begin{remark}
	The number of components of $S_{(k, k)^*}$ is the number of standard Young tableaux of shape $(k, k)^*$.
	This value is $\frac1{k+1}{2k \choose k}$, the \emph{$k$-th Catalan number} (see, e.g., \cite[p. 44]{StanCatalan}).
	From this and \Cref{counting smooths}, the fraction of components of $S_{(k, k)^*}$ that are smooth tends exponentially towards $0$ as $k\to \infty$. 	
\end{remark}

\subsection{Geometry from degree two $\sl_k$ webs} \label{bundled geo}

In the previous subsection, we characterized the smooth components of the Springer fiber $S_{(k, k)^*}$ in terms of degree two $\sl_k$ webs.
The purpose of this subsection is to use these webs to describe the geometry of the smooth components.
These components are \emph{iterated fiber bundles}, as in \Cref{iterated fiber bundle}.

Let $T \in \syt((k, k)^*)$ and write $W_T$ for the corresponding degree two $\sl_k$ web.
We define the \textbf{first claw} of $W_T$ to be the first claw clockwise from the arc $(2k, 1)$ of the boundary disc that does not contain the vertex $2k$.
That is, if $1$ and $2k$ are in different claws, then the first claw is the claw containing $1$; otherwise, it is the first claw strictly clockwise of the claw containing $1$.
The \textbf{second claw} is the next claw appearing clockwise of the first claw, and so on.

\begin{example} \label{first claw}
	First consider the webs in \Cref{intro eg}.
	The web on the left has first claw $\{2, 3, \ldots, 8\}$, while the web on the right has first claw $\{ 1, 2, \ldots, 5 \}$.
	Recall that both webs in \Cref{running eg webs} are in the same \emph{move equivalence class} (as discussed in \Cref{deg 2 sl webs}), and they have the same first claw, namely, the vertices $\{3, 4, 5\}$.
\end{example}

The following is a graph-theoretic reinterpretation of \cite[Theorem 3]{FMSO}.

\begin{theorem} \label{two col geometry}
	Suppose that the degree two $\sl_k$ web $W$ is a forest, hence corresponds to a smooth component $S_W$ of the Springer fiber $S_{(k, k)^*}$.
	Then, $S_W$ is an iterated fiber bundle whose base is given by the following cases.
	\begin{itemize}
		\item 
		If $W$ is disconnected, then let $i$ be the maximum integer such that the vertices $1, 2, \ldots, i$ are in the same claw of $W$.
		The base is
		\begin{equation*}
			\big( \Fl(i) \times \Fl(k) ,~ \mathbb P^i ,~ \mathbb P^{i+1} ,~ \ldots ,~ \mathbb P^{k-1} \big).
		\end{equation*}
		
		\item 
		Otherwise $W$ is connected.
		Let $i$, $j$, and $m = 2k-i-j$ be the number of vertices in the first, second, and third claws of $W$, respectively.
		Let $\ell$ be the multiplicity of the edge between the filled internal vertex and the unfilled vertex of the second claw.
		Then, the base is 
		\begin{equation*}
			\big( \Fl(i) \times \Fl(j) ,~ \Gr_\ell(m) ,~ \mathbb P^{k-m} ,~ \mathbb P^{k-m+1} ,~ \ldots ,~ \mathbb P^{k-1} \big).
		\end{equation*}
	\end{itemize}
\end{theorem}

\begin{proof}
	Let $T \in \syt((k, k)^*)$ be the standard Young tableau for which $W = W_T$ under the construction in \Cref{deg 2 sl webs}.
	Recall that $M_T$ denotes the noncrossing matching corresponding to $T$.
	Throughout this proof, we will use the fact from \Cref{deg 2 sl webs} that $i \in \Asc(T)$ if and only if $1 \le i < 2k$ and $\{i, i+1\}$ is an arc of $M_T$.
	Suppose first that $W$ is disconnected, or equivalently, has exactly two claws.
	By \Cref{counting arcs i--i+1}, we have that $M_T$ has exactly two arcs of the form $\{ i, i+1 \} \mod 2k$.
	We have the following subcases.
	
	If $\{ 1, 2k \}$ is an edge in $M_T$, then $\Asc(T) = \{ \alpha \}$ by \Cref{counting arcs i--i+1}.
	Since the internal vertices of $W$ have the same degree, it follows that $\alpha = k$.
	So $W$ has claws with vertices $\{ 1, \ldots, k \}$ and $\{k+1, \ldots, 2k \}$.
	From \Cref{FMSO triple}, we set $a_T = 0$, $b_T = k$, and $c_T = 0$.
	Then, \Cref{FMSO geometry} guarantees that the component of $S_{(k, k)^*}$ corresponding to $T$ is an iterated fiber bundle with base $\Fl(k)^2$, as desired.
	
	Otherwise, $\Asc(T) = \{ \alpha < \beta \}$.
	The claws in $W_T$ correspond to vertices $\{ \alpha+1, \ldots, \beta \}$ and $\{ \beta+1, \ldots, 2k, 1, \ldots, \alpha \}$.
	Again, the internal vertices have degree $k$, so these vertex sets have equal cardinality, both containing exactly $k$ vertices.
	In particular, $k = \beta - \alpha$.
	Now define $a_T, b_T, c_T$ as in \Cref{FMSO triple} and notice that $i = \alpha$.
	Then,
	\begin{equation*}
		a_T = k - (\beta - \alpha) = 0, \quad 
		a_T + b_T = \alpha = i, \quad 
		b_T + c_T = \beta - \alpha = k, \quad 
		\text{and}\quad 
		b_T = \beta - k = i.
	\end{equation*}
	So from \Cref{FMSO geometry}, we obtain that the component of $S_{(k, k)^*}$ corresponding to $T$ is an iterated bundle with base $\big( \Fl(i) \times \Fl(k) ,~ \mathbb P^i ,~ \mathbb P^{i+1} ,~ \ldots, \mathbb P^{k-1} \big)$, as desired.

	We now treat the second case, when $W$ is connected, so has exactly three claws.
	Again, we have subcases depending on whether or not $\{ 1, 2k \}$ is an edge in the noncrossing matching $M_T$.
	Suppose first that $\{ 1, 2k \}$ is an edge in $M_T$, so by \Cref{counting arcs i--i+1} we have that $\Asc(T) = \{ \alpha < \beta \}$.
	Consequently, the claws in $W_T$ contain the boundary vertices $\{1, \ldots, \alpha \}$, $\{ \alpha+1, \ldots, \beta \}$, and $\{\beta+1, \ldots, 2k \}$.
	Hence, with $i$ and $j$ as in the statement of the theorem, we have that $i = \alpha$ and $j = \beta - \alpha$.
	From this it follows that $m = 2k - \beta$ and $k-m = \beta - k$.
	
	Note that from the construction of the interior of the web in \Cref{deg 2 sl webs}, the value $\ell$ is exactly the number of edges in $M_T$ with one endpoint in each of the vertex sets $\{ 1, \ldots, \alpha\}$ and $\{\beta+1, \ldots, 2k\}$.
	The noncrossing matching $M_T$ is a perfect matching on $2k$ vertices, so has exactly $k$ edges.
	So $\ell$ is the difference between $k$ and the number of edges incident to one of the vertices in $\{ \alpha+1 , \ldots, \beta \}$.
	That is, $\ell = k - (\beta - \alpha)$.
	We have computed that
	\begin{gather*}
		a_T = \ell, \quad 
		a_T + b_T = \alpha = i, \quad 
		b_T + c_T = \beta - \alpha = j,
		\intertext{and,}
		a_T+c_T = 2k - \beta = m, \quad 
		b_T = \beta-k = k-m.
	\end{gather*}
	The desired base for the bundle follows from \Cref{FMSO geometry}.
	
	Lastly, suppose that $\{ 1, 2k \}$ is not an edge in $M_T$, so $\Asc(T) = \{ \alpha < \beta < \gamma \}$ where $\gamma < 2k$.
	The claws in $W_T$ now have corresponding vertices $\{ \alpha+1, \ldots, \beta \}$, $\{ \beta+1, \ldots, \gamma \}$, and $\{ \gamma+1, \ldots, 2k, 1, \ldots, \alpha \}$.
	The sizes $i$ and $j$ of the first two claws are $i = \beta - \alpha$ and $j = \gamma - \beta$, hence $i + j = \gamma - \alpha$.
	Similar to the previous case, we have that $\ell = k - (\gamma - \beta)$.
	So, we have that
	\begin{gather*}
		a_T = k - (\gamma - \beta) = \ell, \quad 
		a_T + b_T = \beta - \alpha = i, \quad 
		b_T + c_T = \gamma - \beta = j, 
		\intertext{and,}
		a_T + c_T = 2k - (\gamma - \alpha) = 2k - i - j, \quad 
		b_T = (\gamma - \alpha) - k = 2k - m - k = k - m.
	\end{gather*}
	The desired result now follows from \Cref{FMSO geometry}.
\end{proof}

\begin{example} \label{geo example}
	Consider the web $W$ in \Cref{forest not tree}.
	This web $W$ has exactly two connected components, so \Cref{two col geometry} says that the component of the Springer fiber $S_{(5, 5)^*}$ corresponding to $W$ is an iterated fiber bundle with base
	\begin{equation*}
		\big( \Fl(3) \times \Fl(5) ,~ \mathbb P^3 ,~ \mathbb P^4 \big).
	\end{equation*}
	
	The component of $S_{(8, 8)^*}$ corresponding to the connected web in the left of \Cref{intro eg} is an iterated fiber bundle with base 
	\begin{equation*}
		\big( \Fl(7) \times \Fl(3) ,~ \Gr_5(6) ,~ \mathbb P^2 ,~ \mathbb P^3 ,~ \ldots ,~ \mathbb P^7 \big).
	\end{equation*}
\end{example}

Every component of a two row Springer fiber is smooth \cite{Fung}.
In this setting, both the geometry and topology are understood, and moreover, we have descriptions of the pairwise intersections of components; see, e.g., \cite{Fung, Russell, SW}.
It would be interesting to obtain, in the two column case, descriptions of the singular components and results analogous to the two row case.

\subsection{Poincar\'e polynomials for smooth components} \label{Poincare}

A degree two $\sl_k$ web inherits a natural action of the dihedral group by rotation and reflection of its underlying graph (leaving the vertex labels fixed).
This corresponds to the procedures of \emph{promotion} and \emph{evacuation}, respectively, on the corresponding standard Young tableau; see \cite{PPR,Fraser2,PatPech}.
In this section, we will show that for smooth components of the Springer fiber $S_{(k, k)^*}$, the Poincar\'e polynomial, defined in \Cref{two col spring}, is invariant under this dihedral action.
Moreover, the Poincar\'e polynomial of a smooth component uniquely identifies the corresponding dihedral orbit of the webs.

\begin{figure}[htbp]
	\begin{tikzpicture}
		\node at (0, 0) {\includegraphics[scale=1.25]{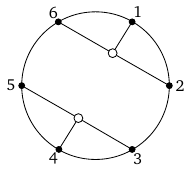}};
		\node at (5.5, 0) {\includegraphics[scale=1.25]{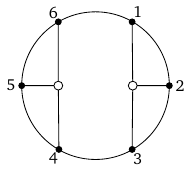}};
	\end{tikzpicture}

	\caption{A web $W$ (left) and its clockwise rotation $\rho\cdot W$ (right).}
	\label{rotation of webs}
\end{figure}

Let $W$ be a degree two $\sl_k$ web that is a forest, and hence corresponds to a smooth component of $S_{(k, k)^*}$, and recall that we denote by $S_W$ this component.
Denote by $P_W(q)$, or simply $P_W$ when there is no ambiguity, the Poincar\'e polynomial of $S_W$.
Also, let $D_{2k} = \langle r, s \mid r^{2k} = s^2 = 1 , rs = sr^{-1} \rangle$ be the dihedral group of a regular $2k$-gon, so $|D_{2k}| = 4k$.
We fix $\rho \in D_{2k}$ to be the rotation such that $\rho\cdot W$ is a clockwise rotation of the underlying graph of $W$; see \Cref{rotation of webs}.

We have the following graph-theoretic reinterpretation of \cite[Theorem 4]{FMSO}.

\begin{theorem} \label{dihedral invariance of smooth components}
	Let $W$ and $W'$ be degree two $\sl_k$ webs that are both forests.
	Then, the Poincar\'e polynomials of $S_W$ and $S_{W'}$ are equal if and only if $W = \sigma \cdot W'$ for some $\sigma \in D_{2k}$.	
\end{theorem}

\begin{remark} \label{iso or homeo components}
	The reflection of a web that fixes the boundary vertex $1$ corresponds to \emph{evacuation} on the associated tableau.
	For a tableau $T$ of any shape, the components of the Springer fiber corresponding to $T$ and its evacuation are isomorphic \cite[Corollary 3.4]{vL} (see also \cite[Theorem 7.4]{KarpPrecup}).
	We obtain by other methods---namely, a direct computation using \Cref{two col geometry}---the equality of Poincar\'e polynomials for components corresponding to evacuation-equivalent tableaux, in the two column rectangle case.
	
	Although this reflection of a web gives an isomorphism of the corresponding Springer fiber components, the same does not hold for arbitrary reflections, nor for a rotation of the web.
	For instance, consider the Springer fiber $S_{(2, 2)}$.
	There are two standard Young tableaux of shape $(2, 2)$; see \Cref{22 tabs}, which also shows that their corresponding webs are in the same rotation orbit.
	Yet \cite[Example 8]{SW} shows that the corresponding components are not isomorphic: The latter is a copy of $\mathbb P^1 \times \mathbb P^1$, while the former is a nontrivial $\mathbb P^1$-bundle over $\mathbb P^1$ (in particular, it is a Hirzebruch surface).
	
	However, these components of $S_{(2, 2)}$ are homeomorphic as topological spaces: 
	All components of a two \emph{row} Springer fiber are homeomorphic, as shown in \cite[Theorem 2.1]{CautKam1}, building on \cite{Khovanov}.
	(For more details on this argument, see \cite[Appendix]{RTtworow}.)
	It is unclear to the author whether the components in the two column rectangle case with equal Poincar\'e polynomials are homeomorphic.	
\end{remark}

\begin{figure}
	\begin{tikzpicture}
		\node at (0, 0) {\includegraphics[scale=1.25]{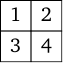}};
		\node at (2.5, 0) {\includegraphics[scale=1.5]{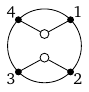}};
		
		\node at (7, 0) {\includegraphics[scale=1.25]{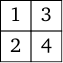}};
		\node at (9.5, 0) {\includegraphics[scale=1.5]{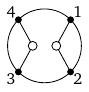}};
	\end{tikzpicture}

	\caption{The standard Young tableaux of shape $(2, 2)$ and their corresponding degree two $\sl_2$ webs.} 
	\label{22 tabs}
\end{figure}

The proof of the forward direction of \Cref{dihedral invariance of smooth components} is a computation that leverages the description of the geometry in \Cref{two col geometry}.
We use in the following proof the Poincar\'e polynomials given in \Cref{basic poincare polys} of projective space, the Grassmannian, and the flag variety.

\begin{proof}[Proof of the forward direction of \Cref{dihedral invariance of smooth components}]
	Suppose that $W = \sigma \cdot W'$ for some $\sigma \in D_{2k}$.
	We first treat the case where the webs $W$ and $W'$ are both disconnected.
	We note that in this case, reflection invariance follows from rotation invariance.
	That is, if $\sigma \in D_{2k}$ is a reflection, then there is some rotation $\tilde \rho \in D_{2k}$ such that $\sigma\cdot W = \tilde \rho \cdot W$.
	Let $W$ be a disconnected web and define $i$ as in \Cref{two col geometry}.
	Then,
	\begin{equation*}
		P_W = [i_W]! ~ [k]! \prod_{j = i_W+1}^k [j] = ([k]!)^2.
	\end{equation*}
	In particular, this does not depend on $i$, which gives the desired rotation invariance for the case of disconnected webs.
	
	We next show that the Poincar\'e polynomial is rotation invariant for webs $W$ and $W'$ that are both connected.
	Again, it is sufficient to consider the case where $W = \rho \cdot W'$.
	We proceed by cases.
	Suppose that the boundary vertex $1$ is in the same claw of $W$ and $W'$, and is not the first vertex in the first claw of $W$. 
	Then the values $i$, $j$, and $\ell$ as in \Cref{two col geometry} are exactly the same for both $W$ and $W'$, and the desired result is immediate.
	This argument also applies in the case that $1$ appears in different claws of $W$ and $W'$.
	
	So suppose that $1$ appears in the same claw of $W$ and $W'$ and is the first vertex of the first claw of $W$.
	Let $i_W, j_W, \ell_W, m_W$ be as in \Cref{two col geometry} for $W$, and similarly $i_{W'}, j_{W'}, \ell_{W'}, m_{W'}$ for $W'$.
	It then follows that $i_W = m_{W'}$, $j_W = i_{W'}$, and $m_W = j_{W'}$.
	Also, $\ell_W = k-j_W$ and $\ell_{W'} = k-j_{W'}$.
	We then compute that
	\begin{align*}
		P_{W'}
		&= [i_{W'}]! ~ [j_{W'}]! \begin{bmatrix} m_{W'} \\ \ell_{W'} \end{bmatrix} ~ [k-m_{W'}+1] \cdots [k-1] [k] 
		\\
		&= [j_W]! ~ [m_W]! \begin{bmatrix} i_W \\ k - m_W \end{bmatrix} ~ [k-i_W+1] \cdots [k-1][k]
		\\[0.25em]
		&= \frac{[i_W]! ~ [j_W]! ~ [m_W]! ~ [k]!}{[k-m_W]! ~ [k-j_W]! ~ [k-i_W]!},
	\end{align*}
	where for the last equality, we use the fact that $i_W - (k - m_W) = k - j_W$ after substituting $m_W = 2k -i_W - j_W$.
	Continuing, and using the observations that $\ell_W = k-j_W$ and
	\begin{equation*}
		k-i_W = (2k - i_W - j_W) - (k-j_W) = m_W - \ell_W,
	\end{equation*}
	the desired result now follows:
	\begin{equation*}
		P_{W'} = \frac{[i_W]! ~ [j_W]! ~ [m_W]! ~ [k]!}{[k-m_W]! ~ [k-j_W]! ~ [k-i_W]!}
		= [i_W]! ~ [j_W]! \begin{bmatrix} 
		m_W \\ \ell_W \end{bmatrix} [k-m_W+1] \cdots [k-1][k]
		= P_W.
	\end{equation*}
	
	Lastly we treat reflection invariance in the connected case.
	Consider any reflection $\sigma$ and take $W$ and $W'$ such that $W = \sigma\cdot W'$.
	Using the rotation invariance, we freely assume that the second claws of $W$ and $W'$ are the same.
	Hence, we have that $i_W = m_{W'}$, $j_W = j_{W'}$, $m_W = i_{W'}$, and $\ell_W = \ell_{W'}$.
	Then,
	\begin{equation*}
		P_W = [i_W]! ~ [j_W]! \begin{bmatrix} m_W \\ \ell_W \end{bmatrix} [k-m_W+1] \cdots [k-1][k] 
		= \frac{[i_W]! ~ [j_W]! ~ [m_W]! ~ [k]!}{[\ell_W]! ~ [m_W - \ell_W]! ~ [k - m_W]!}.
	\end{equation*}
	Since $W$ is $k$-valent, we have that $j_W + \ell_W = k$.
	From this, we observe that
	\begin{equation*}
		k - m_W
		= 2k - m_W - (j_W + \ell_W)
		= (2k - i_{W'} - j_{W'}) - \ell_{W'}
		= m_{W'} - \ell_{W'},
	\end{equation*}
	and analogously, $m_W - \ell_W = k - m_{W'}$.
	With this, we compute that
	\begin{equation*}
		P_W = \frac{[i_{W'}]! ~ [j_{W'}]! ~ [m_{W'}]! ~ [k]!}{[\ell_{W'}]! ~ [m_{W'} - \ell_{W'}]! ~ [k - m_{W'}]!~}
		= [i_{W'}]! ~ [j_{W'}]! \begin{bmatrix} m_{W'} \\ \ell_{W'} \end{bmatrix} [k-m_{W'}+1] \cdots [k-1][k] = P_{W'},
	\end{equation*}
	which completes the proof of the forward direction.
\end{proof}

For the reverse direction, we rely on \Cref{invariant poincare}.
To that end, we require the following lemma that describes the relation between the number of claws of a degree two $\sl_k$ web $W$ and the values $a_T$, $b_T$, and $c_T$ in \Cref{FMSO triple} when $T$ is the tableau associated to $W$.

\begin{lemma} \label{num claws from FMSO triple}
	Let $W$ be a degree two $\sl_k$ web corresponding to a tableau $T$, and assume that $W$ is a forest.
	Define $a_T$ and $c_T$ as in \Cref{FMSO triple}.
	Then, $W$ is connected if and only if both $a_T$ and $c_T$ are nonzero.
\end{lemma}

\begin{proof}
	We showed in the proof of \Cref{two col geometry} that if $W$ is disconnected, then $a_T = 0$.
	So let $W$ be connected, and write $i$ and $j$ for the sizes of its first and second claw, respectively. 
	We also saw in the proof of \Cref{two col geometry} that $a_T = k-j$ and $c_T = k-i$.
	Recall from the first proof of \Cref{counting smooths} that each claw has size at most $k-1$, so both $a_T$ and $c_T$ are strictly positive.	
\end{proof}

\begin{proof}[Proof of the reverse direction of \Cref{dihedral invariance of smooth components}]
	Suppose that we have equality of Poincar\'e polynomials $P_W = P_{W'}$.
	Let $T$ and $T'$ be the standard Young tableaux corresponding to $W$ and $W'$, respectively.
	Define $a_T, b_T, c_T$ and $a_{T'}, b_{T'}, c_{T'}$ as in \Cref{FMSO triple}.
	By \Cref{invariant poincare} we have the following two cases.
	
	Suppose first that $0 \in \{ a_T, c_T \}$ and that $0 \in \{ a_{T'}, c_{T'} \}$.
	\Cref{num claws from FMSO triple} says that both $W$ and $W'$ have exactly two claws, so are in the same $D_{2k}$ orbit.
	
	Otherwise, we have $\{ a_T, b_T, c_T \} = \{ a_{T'} , b_{T'}, c_{T'} \}$ as sets, and that $a_T$, $c_T$, $a_{T'}$, and $c_{T'}$ are all nonzero.
	We show that $W_T$ and $W_{T'}$ have the same multiset of claw sizes, and hence are in the orbit of rotation and reflection.
	
	By \Cref{num claws from FMSO triple}, both $W_T$ and $W_{T'}$ have three claws.
	Recall from \Cref{counting arcs i--i+1} that the web $W_T$ has three claws if and only if
	\begin{itemize}
		\setlength\itemsep{-0.25em}
		\item $|\Asc(T)| = 2$ and there is no $i \in [k-1]$ such that $b_i = 2i$, or,
		\item $|\Asc(T)| = 3$ and there is some $i \in [k-1]$ such that $b_i = 2i$.
	\end{itemize}

	Suppose that $T$ is in the first case, so $\Asc(T) = \{ \alpha < \beta \}$.
	Then---using \Cref{counting arcs i--i+1} again---the first claw of $W_T$ involves boundary vertices $\{1, 2, \ldots, \alpha\}$, and $k-c_T = \alpha$.
	The second claw involves vertices $\{ \alpha+1, \ldots, \beta \}$, and hence has size $\beta - \alpha$.
	Notice that $k-a_T = \beta - \alpha$.
	In particular, the sizes of the claws are determined by $a_T, b_T, c_T$.
	
	Now take $T$ to be in the second case, so $\Asc(T) = \{ \alpha < \beta < \gamma \}$.
	Then from \Cref{counting arcs i--i+1}, the first claw is not incident to the boundary vertex $1$, and instead involves vertices $\{ \alpha+1, \ldots, \beta \}$.
	This claw size is given by $k-c_T = \beta - \alpha$.
	Similarly, the second claw contains vertices $\{ \beta+1, \ldots, \gamma \}$, and $k - a_T = \gamma - \beta$ gives the corresponding size.
	
	Similarly, the sizes of the claws of $W'$ are completely determined by $a_{T'}$, $b_{T'}$, and $c_{T'}$.
	So since $\{ a_T, b_T , c_T \} = \{ a_{T'} , b_{T'} , c_{T'} \}$, we conclude that $W$ and $W'$ have claws of the same sizes, and hence lie in the same dihedral orbit.
	This completes the proof.
\end{proof}

\begin{example}
	Consider the webs in \Cref{webs for 222}.
	Let $P_{W_i}$ be the Poincar\'e polynomial of the component of $S_{(2,2,2)}$ corresponding to the $i$-th web $W_i$ (counted from left-to-right).
	It is a straightforward computation using \Cref{basic poincare polys} and \Cref{two col geometry} to see that
	\begin{equation*}
		P_{W_1} = P_{W_2} = P_{W_3} = ([3]!)^2
		\quad\text{and}\quad
		P_{W_4} = P_{W_5} = [2]^4 [3].
	\end{equation*}
\end{example}

\appendix
\crefalias{section}{appendix}

\section{Geometry from noncrossing matching and ray diagrams} \label{TwoRowWebs}

In \Cref{Geometry}, we used webs to characterize and describe the geometry of the smooth components of two column rectangle Springer fibers.
The results of Fresse, Melnikov, and Sakas-Obeid \cite{FMsingcomp, FMSO} extend to the arbitrary, not necessarily rectangular, two column case.
However, outside of some special cases \cite{PatPechStr, KimFlamingo, Flamingo}, there are no webs corresponding to non-rectangular two column tableaux, so our results do not naturally generalize within the framework of degree two $\sl_k$ webs.

On the other hand, there is a combinatorially natural bijection between two row tableaux and \emph{noncrossing matching and ray} diagrams \cite{Russell, SW}.
These diagrams are used to describe the geometry and topology of two row Springer fibers.
In this appendix, we show that these diagrams can also be used to characterize and describe the geometry of smooth components of two column Springer fibers, even though these diagrams are not two column webs.
This construction relies on tableaux conjugation, which is algebraically subtle.
We believe that this subtlety is reflected in the fact that the characterization and description of the geometry of the smooth components are not as straightforward as those in the two column rectangle and two row cases.

\subsection{Noncrossing matching and ray diagrams} \label{ncm and ray diagrams}

A \textbf{noncrossing matching and ray diagram} with vertices $V = \{1, 2, \ldots, n\}$ and $k$ matching edges is a planar diagram such that the vertices are drawn in order along a horizontal baseline, $2k$ vertices are matched by noncrossing edges drawn above the baseline, and the remaining $n-2k$ vertices each have an incident \emph{ray}.
We draw a ray as an edge with one endpoint that extends off with infinite height, so that for all matching edges $\{ i, \ell \}$ and any vertex $j$ incident to a ray, we have either $j < i$ or $\ell < j$.
That is, a ray cannot appear between the two endpoints of a matching edge.

\begin{figure}[htbp]
	\begin{tikzpicture}
		\node at (0, 0) {\includegraphics{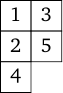}};
		\node at (2.5, 0) {\includegraphics{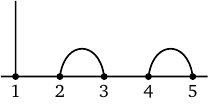}};
		
		\node at (5.5, 0) {\includegraphics{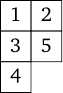}};
		\node at (8, 0) {\includegraphics{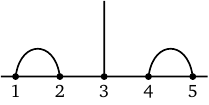}};
		
		\node at (11, 0) {\includegraphics{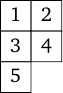}};
		\node at (13.5, 0) {\includegraphics{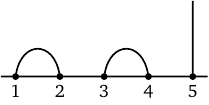}};
	\end{tikzpicture}
	\begin{tikzpicture}
		\node at (0, 0) {\includegraphics{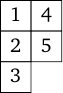}};
		\node at (2.5, 0) {\includegraphics{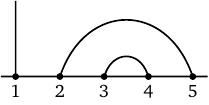}};
		
		\node at (5.5, 0) {\includegraphics{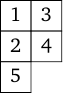}};
		\node at (8, 0) {\includegraphics{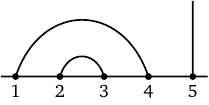}};
	\end{tikzpicture}

	\caption{The five standard Young tableaux of shape $(2, 2, 1)$ and their corresponding noncrossing matching and ray diagrams.}
	\label{221 nmrs}
\end{figure}

The bijection between two column standard Young tableaux and noncrossing matching and ray diagrams is a slight modification to the bijection between two column rectangle standard Young tableaux and noncrossing (perfect) matchings.
The bijection is given as follows.

Fix a tableau $T \in \syt((n-k, k)^*)$, where $1 \le k \le n/2$.
We construct a noncrossing matching and ray diagram with $k$ matching arcs and $n-2k$ rays.
Denote by $\{ b_1 < b_2 < \cdots < b_k \}$ the entries in the second column of $T$.
Match $b_1$ with the largest entry in the first column that is less than $b_1$, then match $b_2$ with the largest remaining unmatched entry in the first column less than $b_2$, and so on.
After completing the matching, draw a ray from each unmatched vertex.

\Cref{221 nmrs} shows the five standard Young tableaux of shape $(3, 2)^*$ and their corresponding noncrossing matching and ray diagrams.

\subsection{Characterization of smooth components}
In this section, we give a characterization of the smooth components of two column Springer fibers in terms of noncrossing matching and ray diagrams, analogous to \Cref{characterization}.
To begin, we record the arbitrary two column characterization of smooth components of Fresse and Melnikov in terms of standard Young tableaux.
Our notation is as in \Cref{two col spring}.
Recall that for a two column tableau $T$, we denote by $\Asc(T)$ the entries $j$ in the first column of $T$ for which $j+1$ is in the second column of $T$.
Also recall that we denote by $b_1 < b_2 < \cdots < b_k$ the entries in the second column of $T$.
The following is the full version of \Cref{FM smooth characterization}.

\begin{theorem}[{\cite[Theorem 1.2]{FMsingcomp}}] \label{full 2 col smooth char}
	Let $T \in \syt((n-k, k)^*)$, where $1 \le k \le n/2$.
	Then, the component $S_T$ of $S_{(n-k, k)^*}$ corresponding to $T$ is smooth if and only if one of the following holds:
	\begin{enumerate}[label=(S\arabic*)]
		\item \label{tau1} $|\Asc(T)| = 1$,
		\item \label{tau2} $|\Asc(T)| = 2$ and either $b_k = n$ or $b_i = 2i$ for some $i \in [k]$, or,
		\item \label{tau3} $|\Asc(T)| = 3$, $b_k = n$, and $b_i = 2i$ for some $i \in [k-1]$.
	\end{enumerate}
\end{theorem}

Let $M$ be a noncrossing matching and ray diagram with $n$ vertices and $n-2k$ rays.
Denote by $S_M$ the component of the Springer fiber corresponding to $M$ under the bijection in \Cref{ncm and ray diagrams} and \Cref{springer components}.
We say that a \textbf{short edge} is an edge of the form $\{ i, i+1\}$ for some $i < n$.
We slightly abuse our terminology to say that a vertex is a ray if and only if the edge incident to it is a ray.

The following is a diagrammatic reinterpretation of \cite[Theorem 1.2]{FMsingcomp}.

\begin{theorem} \label{smooth via ncmrs}
	Let $M$ be a noncrossing matching and ray diagram with $n$ vertices and $n-2k$ rays.
	Then, the component $S_M$ of the Springer fiber $S_{(n-k, k)^*}$ corresponding to $M$ is smooth if and only if
	\begin{enumerate}[label=(S\arabic*$'$)]
		\item \label{consec1} $M$ has exactly 1 short edge;
		\item \label{consec2} $M$ has exactly 2 short edges, but at least one of the vertices $1$, $n$ is not a ray in $M$; or,
		\item \label{consec3} $M$ has exactly 3 short edges, but $\{1, n\}$ is not an edge of $M$, and neither $1$ nor $n$ is a ray.
	\end{enumerate}
\end{theorem}

\begin{proof}
	We show the slightly stronger result that the conditions \ref{consec1}, \ref{consec2}, and \ref{consec3} are equivalent to, respectively, the conditions \ref{tau1}, \ref{tau2}, and \ref{tau3} in \Cref{full 2 col smooth char}.
	Let $T$ be the tableau corresponding to $M$ under the bijection in \Cref{ncm and ray diagrams}.
	In contrast with \Cref{counting arcs i--i+1}, we have that $|\Asc(T)|$ is exactly equal to the number of short edges in $M$.
	So conditions \ref{tau1} and \ref{consec1} are equivalent.
	
	First, note that $n$ appears in the last box in either the first or second column. 
	In particular, $n$ appears in the second column if and only if $n$ is not a ray.
	Let $b_1 < b_2 < \cdots < b_k$ be the entries in the second column of $T$.
	We observed in the proof of \Cref{counting arcs i--i+1} that $b_i = 2i$ for some $i \in [k]$ if and only if the first $i$ rows of $T$ form a standard Young tableau of shape $(i, i)^*$.
	Since $1$ is matched in $M$ if and only if there is such an $i \in [k]$, the conditions \ref{tau2} and \ref{consec2} are equivalent.
	For the third pair of conditions, note that \ref{consec3} prohibits $i = k$, which is equivalent to forbidding $\{1, n\}$ to be an edge in the matching.
	This shows that \ref{tau3} and \ref{consec3} are equivalent, which completes the proof.
\end{proof}

We remark that since matching edges cannot pass over rays, the condition in \ref{consec3} that $\{1, n\}$ is not an edge of $M$ is vacuous unless the corresponding partition is a two column rectangle.

\begin{figure}[htbp]
	\begin{tikzpicture}
		\node at (0, 0) {\includegraphics{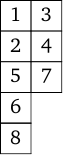}};
		\node at (3.5, 0) {\includegraphics{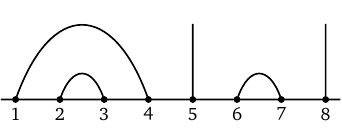}};
		
		\node at (7.75, 0) {\includegraphics{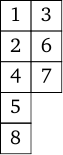}};
		\node at (11.25, 0) {\includegraphics{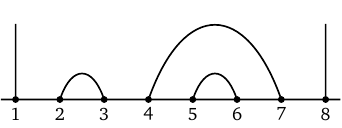}};
	\end{tikzpicture}
	
	\begin{tikzpicture}
		\node at (0, 0) {\includegraphics{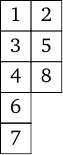}};
		\node at (3.5, 0) {\includegraphics{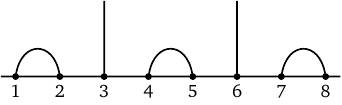}};
		
		\node at (7.75, 0) {\includegraphics{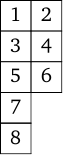}};
		\node at (11.25, 0) {\includegraphics{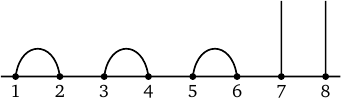}};		
	\end{tikzpicture}
	\caption{Four standard Young tableaux of shape $(5, 3)^*$ and their corresponding noncrossing matching and ray diagrams.}
	\label{some 22211 tableaux}
\end{figure}

\begin{example} \label{ncmr smooth eg}
	All of the noncrossing matching and ray diagrams in \Cref{221 nmrs} correspond to smooth components of the Springer fiber $S_{(3, 2)^*}$.
	In \Cref{some 22211 tableaux}, the components corresponding to the noncrossing matching and ray diagrams on the left are smooth, while those on the right are singular.
	We leave it to the reader to verify with \Cref{full 2 col smooth char} the same conclusions hold from the corresponding tableaux.
\end{example}

\subsection{Geometry from noncrossing matching and ray diagrams}

The goal of this section is to use noncrossing matching and ray diagrams to give a result analogous to \Cref{two col geometry}, which describes the geometry of the smooth components of two column rectangle Springer fibers.
In analogy with the claws of degree two $\sl_k$ webs, whose sizes describe the geometry of smooth components of two column rectangle Springer fibers, we make the following definition.

\begin{definition}
	Let $M$ be a noncrossing matching and ray diagram on $n$ vertices.
	Suppose that $i < j$ are such that $\{ i , i+1 \}$ and $\{ j , j+1 \}$ are short edges in $M$.
	The \textbf{($i$, $j$)-pseudoclaw of $M$} is the set of matching edges and rays of $M$ that do \emph{not} have an endpoint in $\{ i+1, i+2, \ldots, j \}$.
\end{definition}

For an analogy with the degree two $\sl_k$ webs, notice that we have a claw in \Cref{running eg webs} corresponding to the vertices 10, 1, 2, which are exactly the vertices excluded from the set $\{3, 4, \ldots, 9 \}$ in the associated noncrossing matching in \Cref{noncrossing construction}.

\begin{example}
	Consider the noncrossing matching and ray diagrams in \Cref{some 22211 tableaux}.
	The (2, 6)-pseudoclaw in the top-left diagram consists only of the ray from vertex $8$.
	For the lower-left diagram, the (4, 7)-pseudoclaw consists of the matching edge $\{1, 2\}$ and the ray from vertex $3$.
	On the other hand, in the same diagram, the (1, 7)-pseudoclaw is empty.
\end{example}

We give the following diagrammatic reinterpretation of \cite[Theorem 3]{FMSO}.

\begin{theorem} \label{ncmr geometry}
	Let $M$ be a noncrossing matching and ray diagram with $n$ vertices, $k$ matching edges, and $r=n-2k$ rays.
	Suppose that $M$ corresponds to a smooth component $S_M$ of $S_{(n-k, k)^*}$.
	Then $S_M$ is an iterated fiber bundle with base given by the following cases:
	
	\begin{itemize}
		\item 
		If $M$ has exactly one short edge, say $\{i, i+1\}$, then let $\ell$ be the number of rays incident to vertices at most $i$.
		Then, the base is given by
		\begin{equation*}
			\big( \Fl(i) \times \Fl(n-i) ,~ \Gr_\ell(r) \big).
		\end{equation*}
		
		\item 
		Suppose that $M$ has exactly two short edges, denoted $\{ i, i+1 \}$ and $\{ j, j+1 \}$ with $i < j$.
		Let $\ell$ be the size of the ($i$, $j$)-pseudoclaw of $M$.
		If $n$ is a ray, then let $m=r+i$, and the base is
		\begin{equation*}
			\big( \Fl(n-j) \times \Fl(j-i) ,~ \Gr_\ell(m) ,~ \mathbb P^{k-i} ,~ \mathbb P^{k-i+1} ,~ \ldots ,~ \mathbb P^{k-1} \big).
		\end{equation*}
		Otherwise $n$ is not a ray, and let $m = r + (n-j)$, where $n-j$ is the number of vertices strictly larger than $j$.
		The base is
		\begin{equation*}
			\big( \Fl(i) \times \Fl(j-i) ,~ \Gr_\ell(m) ,~ \mathbb P^{j-(r+k)} ,~ \mathbb P^{j-(r+k)+1} ,~ \ldots ,~ \mathbb P^{k-1} \big).
		\end{equation*}
		
		\item 
		If $M$ has exactly three short edges, say $\{ i , i+1 \}$, $\{ j , j+1 \}$, and $\{ h, h+1 \}$, with $i < j < h$, then take $\ell$ to be the size of the ($j$, $h$)-pseudoclaw of $M$.
		Also let $m = n+r-(h-i)$.
		The base is
		\begin{equation*}
			\big( \Fl(j-i) \times \Fl(h-j) ,~ \Gr_\ell(m) ,~  \mathbb P^{(h-i)-(r+k)} ,~ \mathbb P^{(h-i)-(r+k)+1} ,~ \ldots, \mathbb P^{k-1} \big).
		\end{equation*}		
	\end{itemize}
\end{theorem}

We emphasize that each value appearing in the above theorem can be read off the noncrossing matching and ray diagram.
At the same time, the description is not as straightforward as \Cref{two col geometry} in the two column rectangle case.
We view this as a geometric motivation for the development of webs corresponding to arbitrary two column tableaux.

Our proof will again leverage \Cref{FMSO geometry}.
The definitions of $a_T$, $b_T$, and $c_T$ differ slightly from \Cref{FMSO triple} in the rectangular case.

\begin{definition}[{\cite[\textsection 2.4]{FMSO}}] \label{FMSO triple non rect}
	Let $T \in \syts((n-k, k)^*)$ for some $k \le n/2$.
	Define $a_T$, $b_T$, and $c_T$ according to the following cases.
	\begin{itemize}
		\item If $\Asc(T) = \{ \alpha \}$, then $a_T = \alpha - k$, $b_T = k$, and $c_T = (n-k) - \alpha$,

		\item If $\Asc(T) = \{ \alpha < \beta \}$, then
		\begin{itemize}
			\item 
			 if $n \in \col_1(T)$, then $a_T = (n-k) - (\beta - \alpha)$, $b_T = k - \alpha$, and $c_T = \beta - k$,
			
			\item
			if $n \in \col_2(T)$, then $a_T = (n-k) - (\beta - \alpha)$, $b_T = \beta - (n-k)$, and $c_T = (n-k) - \alpha$,
		\end{itemize}
		
		\item If $\Asc(T) = \{ \alpha < \beta < \gamma \}$, then $a_T = (n-k) - (\gamma - \beta)$, $b_T = (\gamma - \alpha) - (n-k)$, and $c_T = (n-k) - (\beta - \alpha)$.
	\end{itemize}
\end{definition}

Like in the rectangular case, each of these values is nonnegative and $b_T$ is strictly positive, but now $a_T + b_T + c_T = n-k$.

\begin{proof}[Proof of \Cref{ncmr geometry}]
	Let $T$ be the standard Young tableau of shape $(n-k, k)^*$ associated to $M$.
	In the proof of \Cref{smooth via ncmrs}, we observed that $|\Asc(T)|$ is equal to the number of short edges in $M$.
	So, we proceed by cases on $|\Asc(T)|$.
	
	Suppose first that $|\Asc(T)| = 1$, so we have exactly one edge $\{ i , i+1 \}$ between consecutive vertices.
	It follows that $\Asc(T) = \{ i \}$.
	From \Cref{FMSO triple non rect}, we set $a_T = i-k$, $b_T = k$, and $c_T = n-k-i$.
	Since $k$ is the number of matching edges, it follows from $|\Asc(T)| = 1$ that the vertices $i, i-1, \ldots, i-(k-1)$, which appear in the first column of $T$, are matched with vertices $i+1, i+2, \ldots, i+k$, respectively.
	In particular, $a_T = i-k$ is exactly the number of vertices less than $i$ that are rays in $M$.
	Now, $a_T + b_T = i$, and $b_T + c_T = n-i$, while $a_T + c_T = n-2k$ is the number of rays in $M$.
	So the desired result follows from \Cref{FMSO geometry}.
	
	Now assume that $|\Asc(T)| = 2$, so $\Asc(T) = \{ i < j \}$.
	Regardless of whether $n$ is a ray or not, we define $a_T = (n-k) - (j-i) = (r+k) - (j-i)$ since $r = n-2k$.
	Since $\{ i, i+1 \}$ and $\{ j, j+1 \}$ are the only short edges, any edge incident to a vertex $v \in \{ i+2, \ldots, j-1 \}$ is either a ray, or a matching edge whose other endpoint $w$ satisfies $w < i$ or $w > j+1$.
	Consequently, $a_T = (r+k)-(j-i)$ is exactly the size $\ell$ of the ($i$, $j$)-pseudoclaw of $M$.

	If $n$ is a ray in $M$, then $n$ appears in the first column of $T$.
	So from \Cref{FMSO triple non rect} we set $b_T = k-i$ and $c_T = j - k$.
	Then, $a_T + b_T = n-j$, $b_T + c_T = j-i$, and $a_T + c_T = r+i$.
	\Cref{FMSO geometry} now gives the claimed formula.
	Otherwise $n$ is not a ray, so $n$ appears in the second column of $T$.
	\Cref{FMSO triple non rect} says to define $b_T = j-(r+k)$ and $c_T = (n-k)-i$, from which it follows that $a_T + b_T = i$, $b_T + c_T = j-i$, and $a_T + c_T = r + (n-j)$.
	The base of the iterated bundle given in \Cref{FMSO geometry} agrees with the desired expression.
	
	Lastly, suppose that $|\Asc(T)| = 3$, so $\Asc(T) = \{ i < j < h \}$ where $h \le n-1$.
	Following \Cref{FMSO triple non rect}, we set $a_T = (n-k) - (h - j) = r+k-(h-j)$, which is exactly $\ell$ as before.
	We also set $b_T = (h - i) - (n-k) = (h-i)-(r+k)$ and $c_T = (n-k) - (j-i)$.
	Then, $a_T + b_T = j-i$, $b_T + c_T = h-j$, and $a_T + c_T = 2(n-k)-(h-i) = n + r - (h-i)$.
	This, together with \Cref{FMSO geometry}, completes the proof.
\end{proof}

\begin{example}
	We saw in \Cref{ncmr smooth eg} that the two diagrams on the left of \Cref{some 22211 tableaux} correspond to smooth components of the Springer fiber $S_{(5,3)^*}$.
	The diagram in the top-left of \Cref{some 22211 tableaux} corresponds to a component whose base is $\big( \Fl(2) \times \Fl(4) , \Gr_1(4) , \mathbb P^1, \mathbb P^2 \big)$, and for the diagram in the lower-left, we have $\big( \Fl(3)^2 , \Gr_2(4) , \mathbb P^1 , \mathbb P^2 \big)$.
\end{example}

\printbibliography

@article{WY,
	author = {Woo, Alexander and Yong, Alexander},
	journal = {Adv. Math.},
	number = {1},
	pages = {205--220},
	title = {When is a {S}chubert variety {G}orenstein?},
	volume = {207},
	year = {2006}}

@book{BT,
	author = {Bott, R. and Tu, L. W.},
	publisher = {Springer-Verlag, New York-Berlin},
	series = {Graduate Texts in Mathematics},
	title = {Differential forms in algebraic topology},
	volume = {82},
	year = {1982}}

@misc{StrikerTenner,
	author = {Striker, J. and Tenner, B.},
	howpublished = {Preprint, available at \url{https://arxiv.org/abs/2603.17287}},
	title = {Forest webs and pattern avoidance},
	year = {2026}}

@article{Jones,
	author = {Jones, V. F. R.},
	journal = {Bull. Amer. Math. Soc. (N.S.)},
	number = {1},
	pages = {103--111},
	title = {A polynomial invariant for knots via von {N}eumann algebras},
	volume = {12},
	year = {1985}}

@article{CMansour,
	author = {Cummings, M. and Mansour, R.},
	journal = {J. Algebraic Combin.},
	number = {36},
	pages = {3 pages},
	title = {{Correction to `Counting two-column Young tableaux corresponding to smooth components of Springer fibers' [J. Algebr. Combin. 61:18 (2025)]}},
	volume = {63},
	year = {2026}}

@misc{KarpPrecup,
	author = {Karp, S. N. and Precup, M. E.},
	howpublished = {Preprint, available at \url{https://arxiv.org/abs/2506.20792}},
	title = {Richardson tableaux and components of {S}pringer fibers equal to {R}ichardson varieties},
	year = {2025}}

@article{PrecupRichmond,
	author = {Precup, M. and Richmond, E.},
	journal = {Trans. Amer. Math. Soc. Ser. B},
	pages = {481--509},
	title = {An equivariant basis for the cohomology of {S}pringer fibers},
	volume = {8},
	year = {2021}}

@misc{KimFlamingo,
	author = {Kim, J.},
	howpublished = {Preprint, available at \url{https://arxiv.org/abs/2402.11994}},
	title = {Rotation invariant webs for three row flamingo {S}pecht modules},
	year = {2024}}

@article{PatPechStr,
	author = {Patrias, R. and Pechenik, O. and Striker, J.},
	journal = {Adv. Math.},
	pages = {Paper No. 108603, 33 pages},
	title = {A web basis of invariant polynomials from noncrossing partitions},
	volume = {408},
	year = {2022}}

@article{Flamingo,
	author = {Fraser, C. and Patrias, R. and Pechenik, O. and Striker, J.},
	journal = {Algebr. Comb.},
	number = {1},
	pages = {235--266},
	title = {Web invariants for flamingo {S}pecht modules},
	volume = {8},
	year = {2025}}

@article{PatPech,
	author = {Patrias, R. and Pechenik, O.},
	journal = {Proc. Amer. Math. Soc. Ser. B},
	pages = {341--352},
	title = {Tableau evacuation and webs},
	volume = {10},
	year = {2023}}

@article{LascouxSchutz,
	author = {Lascoux, A. and Sch\"utzenberger, M.-P.},
	journal = {C. R. Acad. Sci. Paris S\'er. I Math.},
	number = {13},
	pages = {447--450},
	title = {Polyn\^omes de {S}chubert},
	volume = {294},
	year = {1982}}

@article{KMY,
	author = {Knutson, A. and Miller, E. and Yong, A.},
	journal = {J. Reine Angew. Math.},
	pages = {1--31},
	title = {Gr\"obner geometry of vertex decompositions and of flagged tableaux},
	volume = {630},
	year = {2009}}

@misc{White,
	author = {White, D.},
	howpublished = {Personal communication to Brendon Rhoades},
	year = {2007}}

@misc{TymoczkoTalk,
	author = {Tymoczko, J.},
	howpublished = {Talk at ICERM Workshop on Diagrammatic Categorification, \url{https://icerm.brown.edu/video\_archive/4392}},
	title = {Quantum representations and webs},
	year = {2025}}

@article{HopRubey,
	author = {Hopkins, S. and Rubey, M.},
	journal = {Selecta Math. (N.S.)},
	number = {1},
	pages = {Paper No. 10, 38 pages},
	title = {Promotion of {K}reweras words},
	volume = {28},
	year = {2022}}

@article{CautKam1,
	author = {Cautis, S. and Kamnitzer, J.},
	journal = {Duke Math. J.},
	number = {3},
	pages = {511--588},
	title = {Knot homology via derived categories of coherent sheaves. {I}. {T}he {${\mathfrak{sl}}(2)$}-case},
	volume = {142},
	year = {2008}}

@book{SpaltBook,
	author = {Spaltenstein, N.},
	publisher = {Springer-Verlag, Berlin-New York},
	series = {Lecture Notes in Mathematics},
	title = {Classes unipotentes et sous-groupes de {B}orel},
	volume = {946},
	year = {1982}}

@phdthesis{FungThesis,
	author = {Fung, F. Y. C.},
	school = {Princeton University},
	title = {On the relation between {S}pringer fibers of the general linear group and {K}azhdan-{L}usztig theory},
	year = {1997}}

@misc{OEIS,
	author = {{OEIS Foundation Inc.}},
	howpublished = {Published electronically at \url{https://oeis.org}},
	options = {giveninits=false},
	shorthand = {OEIS},
	title = {{The On-Line Encyclopedia of Integer Sequences}},
	year = {2025}}

@article{Anderson,
	author = {Anderson, D.},
	journal = {Adv. Math.},
	pages = {440--485},
	title = {{$K$}-theoretic {C}hern class formulas for vexillary degeneracy loci},
	volume = {350},
	year = {2019}}

@article{GreenKL,
	author = {Green, R. M.},
	journal = {J. Algebraic Combin.},
	number = {2},
	pages = {165--171},
	title = {Leading coefficients of {K}azhdan-{L}usztig polynomials and fully commutative elements},
	volume = {30},
	year = {2009}}

@article{Stembridge,
	author = {Stembridge, J. R.},
	journal = {J. Algebraic Combin.},
	number = {4},
	pages = {353--385},
	title = {On the fully commutative elements of {C}oxeter groups},
	volume = {5},
	year = {1996}}

@article{BHY,
	author = {Billey, S. C. and Holroyd, A. E. and Young, B.},
	journal = {Algebr. Comb.},
	number = {2},
	pages = {217--248},
	title = {A bijective proof of {M}acdonald's reduced word formula},
	volume = {2},
	year = {2019}}

@article{BJS,
	author = {Billey, S. C. and Jockusch, W. and Stanley, R. P.},
	journal = {J. Algebraic Combin.},
	number = {4},
	pages = {345--374},
	title = {Some combinatorial properties of {S}chubert polynomials},
	volume = {2},
	year = {1993}}

@incollection{AbeBilley,
	author = {Abe, H. and Billey, S.},
	booktitle = {Schubert {C}alculus---{O}saka 2012},
	pages = {1--52},
	publisher = {Math. Soc. Japan, [Tokyo]},
	series = {Adv. Stud. Pure Math.},
	title = {Consequences of the {L}akshmibai-{S}andhya theorem: the ubiquity of permutation patterns in {S}chubert calculus and related geometry},
	volume = {71},
	year = {2016}}

@article{LakSand,
	author = {Lakshmibai, V. and Sandhya, B.},
	journal = {Proc. Indian Acad. Sci. Math. Sci.},
	number = {1},
	pages = {45--52},
	title = {Criterion for smoothness of {S}chubert varieties in {${\rm Sl}(n)/B$}},
	volume = {100},
	year = {1990}}

@book{StanCatalan,
	author = {Stanley, R. P.},
	publisher = {Cambridge University Press, New York},
	title = {Catalan Numbers},
	year = {2015}}

@article{GPPSSprom,
	author = {Gaetz, C. and Pechenik, O. and Pfannerer, S. and Striker, J. and Swanson, J. P.},
	journal = {Comb. Theory},
	number = {2},
	pages = {Paper No. 15, 56 pages},
	title = {Promotion permutations for tableaux},
	volume = {4},
	year = {2024}}

@article{PPR,
	author = {Petersen, T. K. and Pylyavskyy, P. and Rhoades, B.},
	journal = {J. Algebraic Combin.},
	number = {1},
	pages = {19--41},
	title = {Promotion and cyclic sieving via webs},
	volume = {30},
	year = {2009}}

@article{KungRota,
	author = {Kung, J. P. S. and Rota, G.-C.},
	journal = {Bull. Amer. Math. Soc. (N.S.)},
	number = {1},
	pages = {27--85},
	title = {The invariant theory of binary forms},
	volume = {10},
	year = {1984}}

@article{RTW,
	author = {Rumer, Y. and Teller, E. and Weyl, H.},
	journal = {Nachr. Ges. Wiss. G\"ottingen Math.-Phys. Kl.},
	pages = {499--504},
	title = {Eine f\"ur die Valenztheorie geeignete Basis der bin\"aren Vektorinvarianten},
	year = {1932}}

@article{TL,
	author = {Temperley, H. N. V. and Lieb, E. H.},
	journal = {Proc. Roy. Soc. London Ser. A},
	number = {1549},
	pages = {251--280},
	title = {Relations between the `percolation' and `colouring' problem and other graph-theoretical problems associated with regular planar lattices: some exact results for the `percolation' problem},
	volume = {322},
	year = {1971}}

@article{LS,
	author = {L\^{e}, T. T. Q. and Sikora, A. S.},
	journal = {J. Topol.},
	number = {3},
	pages = {Paper No. e12350, 93 pages},
	title = {Stated {${\rm SL}(n)$}-skein modules and algebras},
	volume = {17},
	year = {2024}}

@article{IK,
	author = {Ishibashi, T. and Kano, S.},
	journal = {Math. Z.},
	number = {4},
	pages = {Paper No. 88, 53 pages},
	title = {Unbounded {$\mathfrak{sl}_3$}-laminations around punctures},
	volume = {310},
	year = {2025}}

@article{SSW,
	author = {Shen, L. and Sun, Z. and Weng, D.},
	journal = {Trans. Amer. Math. Soc.},
	number = {8},
	pages = {5513--5549},
	title = {Intersections of dual {${\rm SL}_3$}-webs},
	volume = {378},
	year = {2025}}

@article{KhovanovLinkInv,
	author = {Khovanov, M.},
	journal = {Algebr. Geom. Topol.},
	pages = {1045--1081},
	title = {sl(3) link homology},
	volume = {4},
	year = {2004}}

@article{DKS,
	author = {Douglas, D. C. and Kenyon, R. and Shi, H.},
	journal = {Trans. Amer. Math. Soc.},
	number = {2},
	pages = {921--950},
	title = {Dimers, webs, and local systems},
	volume = {377},
	year = {2024}}

@incollection{GNST,
	author = {Goldwasser, T. and Nadeem, M. and Sun, G. and Tymoczko, J.},
	booktitle = {{Advances in the Mathematical Sciences}},
	pages = {31--81},
	publisher = {Springer, Cham},
	series = {Assoc. Women Math. Ser.},
	title = {{Cell Closures for Two-Row Springer Fibers via Noncrossing Matchings}},
	volume = {38},
	year = {2025}}

@article{CHY,
	author = {Chen, X. and He, W. and Yu, S.},
	journal = {Represent. Theory},
	pages = {616--658},
	title = {Quantization of the minimal nilpotent orbits and the quantum {H}ikita conjecture},
	volume = {29},
	year = {2025}}

@unpublished{Postnikov,
	author = {Postnikov, A.},
	note = {Unpublished, available at \href{https://arxiv.org/abs/math/0609764}{arXiv:0609764}},
	title = {Total positivity, {G}rassmannians, and networks},
	year = {(2006), 79 pages}}

@book{Sagan,
	author = {Sagan, B. E.},
	publisher = {American Mathematical Society, Providence, RI},
	series = {Graduate Studies in Mathematics},
	title = {Combinatorics: The Art of Counting},
	volume = {210},
	year = {2020}}

@article{Fraser2,
	author = {Fraser, C.},
	journal = {Comb. Theory},
	pages = {Paper No. 11, 26 pages},
	title = {Webs and canonical bases in degree two},
	volume = {3},
	year = {2023}}

@article{GPPSS4,
	author = {Gaetz, C. and Pechenik, O. and Pfannerer, S. and Striker, J. and Swanson, J. P.},
	journal = {Invent. Math.},
	pages = {703--804},
	title = {Rotation-invariant web bases from hourglass plabic graphs},
	volume = {243},
	year = {2026}}

@article{BodishWu,
	author = {Bodish, E. and Wu, H.},
	journal = {Adv. Math.},
	pages = {Paper No. 110514, 65 pages},
	title = {Webs for the quantum orthogonal group},
	volume = {480},
	year = {2025}}

@article{Bodish,
	author = {Bodish, E.},
	journal = {Quantum Topol.},
	number = {3},
	pages = {407--458},
	title = {Web calculus and tilting modules in type {$C_2$}},
	volume = {13},
	year = {2022}}

@article{FLL,
	author = {Fraser, C. and Lam, T. and Le, I.},
	journal = {Trans. Amer. Math. Soc.},
	number = {9},
	pages = {6087--6124},
	title = {From dimers to webs},
	volume = {371},
	year = {2019}}

@article{Lam,
	author = {Lam, T.},
	journal = {J. Lond. Math. Soc. (2)},
	number = {3},
	pages = {633--656},
	title = {Dimers, webs, and positroids},
	volume = {92},
	year = {2015}}

@article{FominP,
	author = {Fomin, S. and Pylyavskyy, P.},
	journal = {Adv. Math.},
	pages = {717--787},
	title = {Tensor diagrams and cluster algebras},
	volume = {300},
	year = {2016}}

@article{FraserP,
	author = {Fraser, C. and Pylyavskyy, P.},
	journal = {Adv. Math.},
	pages = {Paper No. 108796, 83 pages},
	title = {Tensor diagrams and cluster combinatorics at punctures},
	volume = {412},
	year = {2023}}

@article{Higgins,
	author = {Higgins, V.},
	journal = {Quantum Topol.},
	number = {1},
	pages = {1--63},
	title = {Triangular decomposition of {${\rm SL}_3$} skein algebras},
	volume = {14},
	year = {2023}}

@article{FKK,
	author = {Fontaine, B. and Kamnitzer, J. and Kuperberg, G.},
	journal = {Compos. Math.},
	number = {11},
	pages = {1871--1912},
	title = {Buildings, spiders, and geometric {S}atake},
	volume = {149},
	year = {2013}}

@article{PS,
	author = {Perrin, N. and Smirnov, E.},
	journal = {Bull. Soc. Math. France},
	number = {3},
	pages = {309--333},
	title = {Springer fiber components in the two columns case for types {$A$} and {$D$} are normal},
	volume = {140},
	year = {2012}}

@unpublished{Wehrli,
	author = {Wehrli, S. M.},
	note = {Unpublished, available at \href{https://arxiv.org/abs/0908.2185}{arXiv:0908.2185}},
	title = {A remark on the topology of {$(n, n)$ S}pringer varieties},
	year = {(2009), 8 pages}}

@article{ILW,
	author = {Im, M. S. and Lai, C.-J. and Wilbert, A.},
	journal = {J. Algebra},
	pages = {217--248},
	title = {A study of irreducible components of {S}pringer fibers using quiver varieties},
	volume = {591},
	year = {2022}}

@article{Nevins,
	author = {Nevins, M.},
	journal = {Pacific J. Math},
	number = {2},
	pages = {259--301},
	title = {The local character expansion as branching rules: nilpotent cones and the case of {${\rm SL}(2)$}},
	volume = {329},
	year = {2024}}

@article{Carrell,
	author = {Carrell, J. B.},
	journal = {J. Differential Geom.},
	number = {3},
	pages = {651--668},
	title = {Bruhat cells in the nilpotent variety and the intersection rings of {S}chubert varieties},
	volume = {37},
	year = {1993}}

@article{MV,
	author = {Mirkovi\'c, I. and Vybornov, M.},
	journal = {Adv. Math.},
	note = {With an appendix by Vasily Krylov},
	pages = {Paper No. 108397, 54 pages},
	title = {Comparison of quiver varieties, loop {G}rassmannians and nilpotent cones in type {$A$}},
	volume = {407},
	year = {2022}}

@article{evenmoreKostant,
	author = {Kostant, B.},
	journal = {Ann. of Math. (2)},
	pages = {72--144},
	title = {Lie algebra cohomology and generalized {S}chubert cells},
	volume = {77},
	year = {1963}}

@article{moreKostant,
	author = {Kostant, B.},
	journal = {Ann. of Math. (2)},
	pages = {329--387},
	title = {Lie algebra cohomology and the generalized {B}orel-{W}eil theorem},
	volume = {74},
	year = {1961}}

@article{Kostant,
	author = {Kostant, B.},
	journal = {Amer. J. Math.},
	pages = {327--404},
	title = {Lie group representations on polynomial rings},
	volume = {85},
	year = {1963}}

@article{Kuperberg,
	author = {Kuperberg, G.},
	journal = {Comm. Math. Phys.},
	number = {1},
	pages = {109--151},
	title = {Spiders for rank {$2$} {L}ie algebras},
	volume = {180},
	year = {1996}}

@article{RTtworow,
	author = {Russell, H. M. and Tymoczko, J. S.},
	journal = {Math. Proc. Cambridge Philos. Soc.},
	number = {1},
	pages = {59--81},
	title = {Springer representations on the {K}hovanov {S}pringer varieties},
	volume = {151},
	year = {2011}}

@article{FMsmoothchar,
	author = {Fresse, L. and Melnikov, A.},
	journal = {Selecta Math. (N.S.)},
	number = {3},
	pages = {393--418},
	title = {On the singularity of the irreducible components of a {S}pringer fiber in {$\mathfrak{sl}_n$}},
	volume = {16},
	year = {2010}}

@article{Fresse222,
	author = {Fresse, L.},
	journal = {C. R. Math. Acad. Sci. Paris},
	number = {11-12},
	pages = {631--636},
	title = {Composantes singuli\`eres des fibres de {S}pringer dans le cas deux-colonnes},
	volume = {347},
	year = {2009}}

@article{Vargas,
	author = {Vargas, J. A.},
	journal = {Bol. Soc. Mat. Mex.},
	pages = {1--14},
	title = {Fixed points under the action of unipotent elements of {$SL_n$} in the flag variety},
	volume = {24},
	year = {1979}}

@article{vL,
	author = {{van Leeuwen}, M. A. A.},
	journal = {J. Algebra},
	number = {2},
	pages = {397--426},
	shorthand = {vLee00},
	title = {Flag varieties and interpretations of {Y}oung tableau algorithms},
	volume = {224},
	year = {2000}}

@article{Spaltenstein,
	author = {Spaltenstein, N.},
	journal = {Indag. Math.},
	number = {5},
	pages = {452--456},
	title = {The fixed point set of a unipotent transformation on the flag manifold},
	volume = {38},
	year = {1976}}

@article{SeidSmith,
	author = {Seidel, P. and Smith, I.},
	journal = {Duke Math. J.},
	number = {3},
	pages = {453--514},
	title = {A link invariant of the symplectic geometry of nilpotent slices},
	volume = {134},
	year = {2006}}

@article{Khovanov,
	author = {Khovanov, M.},
	journal = {Commmun. Contemp. Math},
	number = {4},
	pages = {561--577},
	title = {Crossingless matchings and the cohomology of {$(n, n)$} {S}pringer varieties},
	volume = {6},
	year = {2004}}

@article{KhovanovTangles,
	author = {Khovanov, M.},
	journal = {Algebr. Geom. Topol.},
	pages = {665--741},
	title = {A functor-valued invariant of tangles},
	volume = {2},
	year = {2002}}

@incollection{BM,
	author = {Borho, W. and MacPherson, R.},
	booktitle = {Analysis and topology on singular spaces, {II}, {III} ({L}uminy, 1981)},
	pages = {23--74},
	publisher = {Soc. Math. France, Paris},
	series = {Ast\'erisque},
	title = {Partial resolutions of nilpotent varieties},
	volume = {101-102},
	year = {1983}}

@article{Lusztig,
	author = {Lusztig, G.},
	journal = {Adv. Math.},
	number = {2},
	pages = {169--178},
	title = {Green polynomials and singularities of unipotent classes},
	volume = {42},
	year = {1981}}

@book{Slodowy,
	author = {Slodowy, P.},
	series = {Communications of the Mathematical Institute, Rijksuniversiteit Utrecht},
	title = {Four lectures on simple groups and singularities},
	volume = {11},
	year = {1980}}

@article{SpringerTrig,
	author = {Springer, T. A.},
	journal = {Invent. Math.},
	pages = {173--207},
	title = {Trigonometric sums, Green functions of finite groups and representations of Weyl groups},
	volume = {36},
	year = {1976}}

@incollection{SpringerRes,
	author = {Springer, T. A.},
	booktitle = {Algebraic {G}eometry ({I}nternat. {C}olloq., {T}ata {I}nst. {F}und. {R}es., {B}ombay, 1968)},
	pages = {373--391},
	series = {Tata Inst. Fund. Res., Bombay},
	title = {The unipotent variety of a semi-simple group},
	volume = {4},
	year = {1969}}

@incollection{TymoczkoSurvey,
	author = {Tymoczko, J.},
	booktitle = {Around {L}anglands correspondences},
	pages = {359--376},
	publisher = {Amer. Math. Soc., Providence, RI},
	series = {Contemp. Math.},
	title = {The geometry and combinatorics of {S}pringer fibers},
	volume = {691},
	year = {2017}}

@article{Russell,
	author = {Russell, H. M.},
	journal = {Pacific J. Math},
	number = {1},
	pages = {221--255},
	title = {A topological construction for all two-row {S}pringer varieties},
	volume = {253},
	year = {2011}}

@article{SW,
	author = {Stroppel, C. and Webster, B.},
	journal = {Comment. Math. Helv.},
	number = {2},
	pages = {477--520},
	title = {2-block {S}pringer fibers: convolution algebras and coherent sheaves},
	volume = {87},
	year = {2012}}

@article{Fung,
	author = {Fung, F. Y. C.},
	journal = {Adv. Math.},
	number = {2},
	pages = {244--276},
	title = {On the topology of components of some {S}pringer fibers and their relation to {K}azhdan-{L}usztig theory},
	volume = {178},
	year = {2003}}

@article{Mansour,
	author = {Mansour, R.},
	journal = {J. Algebraic Combin.},
	number = {18},
	pages = {9 pages},
	title = {Counting two-column {Y}oung tableaux corresponding to smooth components of {S}pringer fibers},
	volume = {61},
	year = {2025}}

@article{GPPSS,
	author = {Gaetz, C. and Pechenik, O. and Pfannerer, S. and Striker, J. and Swanson, J. P.},
	journal = {Int. Math. Res. Not. IMRN},
	number = {13},
	pages = {rnaf189},
	title = {Web bases in degree two from hourglass plabic graphs},
	year = {2025}}

@article{FMSO,
	author = {Fresse, L. and Melnikov, A. and Sakas-Obeid, S.},
	journal = {Proc. Amer. Math. Soc.},
	number = {6},
	pages = {2301--2315},
	shorthand = {FMSO15},
	title = {On the structure of smooth components of {S}pringer fibers},
	volume = {143},
	year = {2015}}

@article{FMsingcomp,
	author = {Fresse, L. and Melnikov, A.},
	journal = {Algebr. Represent. Theory},
	number = {6},
	pages = {1063--1086},
	title = {Some characterizations of singular components of {S}pringer fibers in the two-column case},
	volume = {14},
	year = {2011}}

\end{document}